\documentclass[12pt,reqno]{amsart}
\usepackage{amsfonts}
\usepackage{a4wide}

\numberwithin{equation}{section}

\newtheorem{theorem}{Theorem}[section]
\newtheorem{proposition}[theorem]{Proposition}
\newtheorem{corollary}[theorem]{Corollary}
\newtheorem{lemma}[theorem]{Lemma}

\theoremstyle{definition}

\newtheorem{definition}[theorem]{Definition}

\theoremstyle{remark}
\newtheorem{remark}[theorem]{Remark}

\begin{document}

\title[ Scattering solutions
 ]
{  Scattering for the focusing $L^{2}$ -supercritical and $\dot H^2$-subcritical biharmonic
    NLS Equations
}

 \author{Qing Guo
}

\address{College of Science, Minzu University of China, Beijing 100081, China\\}

\email{guoqing@amss.ac.cn}





\begin{abstract}
We consider the focusing $\dot H^{s_c}$-critical biharmonic Schr\"odinger
equation, and prove a global wellposedness and scattering result for 
the radial  data $u_0\in H^2(\mathbb R^N)$ satisfying  $ M(u_0)^{\frac{2-s_c}{s_c}}E(u_0)<M(Q)^{\frac{2-s_c}{s_c}}E(Q)
$ and  $ \|u_{0}\|^{\frac{2-s_c}{s_c}}_{2}\|\Delta
u_{0}\|_{2}<\|Q\|^{\frac{2-s_c}{s_c}}_{2}\|\Delta Q\|_{2}, $  where $s_c\in(0,2)$ and $Q$ is the ground state of
$\Delta^2Q+(2-s_c)Q-|Q|^{p-1}Q=0$.
\end{abstract}

\maketitle


\section{Introduction}

The biharmonic Schr\"odinger equations, which are also called the fourth-order Schr\"odinger equations,
\begin{equation}\label{0}
 iu_{t}+\Delta^2 u-\varepsilon\Delta u+f(|u|^{2})u=0
                          \end{equation}
                        with $\varepsilon=\pm1$ or $\varepsilon=0$
were introduced by Karpman \cite{Karp} and Karpman and Shagalov \cite{Ka-Sh} to take into account the
role of small fourth-order dispersion terms in the propagation of intense laser beams in a bulk medium
with Kerr nonlinearity. Sharp dispersive estimates for the biharmonic Schr\"odinger operator
have recently been obtained in \cite{Be-Ko-Sa}, while specific nonlinear biharmonic Schr\"odinger equations
as  \eqref{0} were discussed in \cite{Guo-Wang,H-H-W,Fi-Il-Pa,Se}.
Related equations also appeared in \cite{F-I-S,H-J,Se2}.
 For a pure power-type  nonlinearity, i.e., $f(|u|^{2})u=\mu|u|^{p-1}u$, the equation \eqref{0} is subcritical
 for $N\leq4$ or for $p<1+\frac{8}{N-4}$ ($N>4$). When $N\geq5$, criticality in the energy space $H^2(\mathbb R^N)$
  appears with the power
 $p=1+\frac{8}{N-4}$. Fibich, Ilan and Papanicolaou in \cite{Fi-Il-Pa} describe various properties of the equation in the
 subcritical regime, with part of their analysis relying on very interesting numerical developments.
 Segata in \cite{Se}  proved scattering in $\mathbb R^1$ for the cubic nonlinearity; while in higher dimensions,
 the scattering results were obtained in \cite{MiaoWuZhang,Pa3}.
 Global well-posedness and scattering for the energy critical case
 was considered in \cite{Pa1,Pa4,MiaoXuZhao,MiaoXu}, while in \cite{PS}, the authors
proved the same results for the mass-critical fourth-order Schr\"odinger
equation in high dimensions. As discussed in \cite{Pa1}, the scattering results for the subcritical defocusing case, i.e.,
for $f(|u|^{2})u=|u|^{p-1}u$
 with $1+\frac{8}{N}<p<1+\frac{8}{N-4}$,  could  be obtained following the strategy in Lin and Strauss \cite{Lin-Str},
 see also \cite{TC}.
However, to the authors' knowledge,  there have not been any scattering
 results for the focusing case ($f(|u|^{2})u=-|u|^{p-1}u$) in the subcritical regime.

In this paper, we
  consider the focusing $L^{2}$ -supercritical and $\dot H^2$-subcritical biharmonic
    nonlinear Schr\"odinger equation
\begin{equation}\label{1.1}
\left\{ \begin{aligned}
         \ iu_{t}+\Delta^2 u-|u|^{p-1}u&=0,\ \ \ (x,t)\in \mathbb R^{N}\times \mathbb R, \\
                  \ u(x,0)&=u_{0}(x)\in H^{2}(\mathbb R^{N}),
                          \end{aligned}\right.
                          \end{equation}
where $u(x,t)$ is a complex-valued function in $\mathbb R^N\times\mathbb R$ with the space dimension $N$
satisfying $1+\frac{8}{N}<p<1+\frac{8}{N-4}$ (when $N\leq 4$, $1+\frac{4}{N}<p<\infty$).
Equation  \eqref{1.1} admits two important conservation laws in energy space  ~$H^{2}(\mathbb R^N):$
\begin{align*}
Mass:\ \ \ \ M(u)(t)&\equiv \int|u(x,t)|^{2}dx=M(u_{0});\\
Energy:\ \ \ \ E(u)(t)&\equiv \frac{1}{2}\int|\Delta u(x,t)|^{2}dx-\frac{1}{p+1}\int|u(x,t)|^{p+1}dx=E(u_{0}).
\end{align*}
Moreover, it is easy to check that equation \eqref{1.1} is invariant under the
scaling~$u(x,t)\rightarrow\lambda^\frac{4}{p-1}u(\lambda
x,\lambda^2t)$  which also leaves the  norm of the homogeneous
Sobolev space  $\dot{H}^{s_c}(\mathbb R^{N})$  invariant,
where $s_c\in(0,2)$ is defined by $s_c=\frac{N}{2}-\frac{4}{p-1}$.
 That is why we also call this equation  $\dot H^{s_c}$-critical.
Other scaling invariant quantities are  $\|\Delta u\|_{L^2(\mathbb R^N)}\|u\|_{L^2(\mathbb R^N)}^{\frac{2-s_{c}}{s_{c}}}$
and  $E(u)M(u)^{\frac{2-s_{c}}{s_{c}}}.$

Equation \eqref{1.1} possesses a focusing nonlinearity ($f(|u|^{2})u=-|u|^{p-1}u$), so
 one cannot hope to a similar global result as in \cite{Pa1}.
Indeed, the existence of a nontrivial solution of the elliptic equation
\begin{align}\label{Q}
\Delta^2Q+(2-s_c)Q-|Q|^{p-1}Q=0,
\end{align}
which we refer to as ground state $Q\in H^2(\mathbb R^N)$, can be obtained by similar method to that used in \cite{lenzmann}.
We then conclude that solitary waves $u(x,t)=e^{i(2-s_c)t}Q(x)$  do not scatter.
One can refer to \cite{Fi-Il-Pa} for some similar
results.

Our aim in this paper is to obtain the following result of  scattering   for the  solutions of \eqref{1.1}
with radial data. This scattering result would complement the very recent analysis of Boulenger and Lenzmann \cite{lenzmann}.

\begin{theorem}\label{th1}
Let $u_0\in H^2$ be radial and let $u$ be the corresponding solution
to \eqref{1.1} with maximal forward time interval of existence $I\subset\mathbb R$.
Suppose $ M(u_{0})^{\frac{2-s_c}{s_c}}E(u_{0})<M(Q)^{\frac{2-s_c}{s_c}}E(Q),
$ where $Q$ is the solution of \eqref{Q}. If $$ \|u_{0}\|^{\frac{2-s_c}{s_c}}_{L^2(\mathbb R^N)}\|\Delta
u_{0}\|_{L^2(\mathbb R^N)}<\|Q\|^{\frac{2-s_c}{s_c}}_{L^2(\mathbb R^N)}\|\Delta Q\|_{L^2(\mathbb R^N)}, $$ then
$I=(-\infty,+\infty)$, and $u$ scatters in $H^2(\mathbb R^N)$. That is there
exists $\phi_\pm\in H^2(\mathbb R^N)$ such that
$\lim_{t\rightarrow\pm\infty}\|u(t)-e^{it\Delta^2}\phi_\pm\|_{H^2(\mathbb R^N)}=0$.
\end{theorem}

Our paper is organized as follows. We fix notations in the end of section 1. In section 2,
We recall the local theory for \eqref{1.1} established by \cite{Pa1}.
After that, we introduce the inhomogeneous Strichartz's estimates, upon
which we sketch the proof of the small data scattering and the perturbation theory.
The variational structure of the ground state of an elliptic problem is given in section 3.
In section 4, we prove a dichotomy proposition of global well-posedness versus blowing up, which
yields the comparability of the total  energy and the kinetic energy.
The concentration compactness principle is used in section 5 to give a critical element, which
yields a contradiction through a virial-type estimate in section 6, concluding the proof of Theorem \ref{th1}.

\textbf{Notations:} In what follows, we denote by $c$ a generic constant that is allowed to depend on $N$ and $p$.
The exact value of that constant may change from one to another. We write  $C_a$ when there is more dependence.
We let $L^r=L^r(\mathbb R^N)$ be the usual Lebesgue spaces with the norm
defined by $\|\cdot\|_r$,  and $L^q(I,L^r)$ be the spaces of measurable functions from
an interval $I\subset \mathbb R$ to $L^q$ whose $L^q(I,L^r)$-norm is finite, where
$L^q(I,L^r)=\left(\int_I\|u(t)\|_{r}^qdt\right)^{\frac1q}$.
Moreover, we define the Fourier transform  on $\mathbb{R}^{N}$ by
 $\hat{f}(\xi)=(2\pi)^{-N/2}\int e^{-ix\xi}f(x)dx$.
 For $s\in \mathbb R$,
the pseudo-differential operator $(-\Delta)^s$ (or denoted by $|\nabla|^{2s}$) is defined by
$\widehat{(-\Delta)^sf}(\xi)\equiv \mathcal |\xi|^{2s}\hat{f}(\xi),$
which in turn defines the
 homogeneous Sobolev space $\dot{H}^{s}=\dot{H}^{s}(\mathbb{R}^{N})
 \equiv\left\{f\in\mathcal S'(\mathbb R^N): \int|\xi|^{2s}|\hat{f}(\xi)|^{2}d\xi<\infty
\right\}$
with its norm defined by $\|f\|_{\dot{H}^{s}}=\|(-\Delta)^sf\|_2$,
where $\mathcal S'(\mathbb R^N)$ denotes the space of tempered distributions.

\section{Local theory and Strichartz estimates }

We start in this section by recalling   the Strichartz estimates
 established by Pausader \cite{Pa1}.
 We say a pair $(q,r)$ is Schr\"odinger admissible, for short S-admissible, if
 $2\le q,r\le\infty$, $(q,r,N)\ne (2,\infty,2)$, and
 $\frac{2}{q}+\frac{N}{r}=\frac{N}{2}$.
 Also we use the terminology that a pair $(q,r)$ is biharmonic  admissible, for short B-admissible, if
 $2\le q,r\le\infty$, $(q,r,N)\ne (2,\infty,4)$, and
 $\frac{4}{q}+\frac{N}{r}=\frac{N}{2}$.

 The Strichartz estimates are stated as follows.
 Let $u\in C(I,H^{-4}(\mathbb{R}^N))$ be a solution of
 \begin{align}\label{eq}
 u(t)=e^{it(\Delta^2+\varepsilon\Delta)}u_0+i\int_0^te^{i(t-s)(\Delta^2+\varepsilon\Delta)}h(s)ds
 \end{align}
 with $\varepsilon\in\{-1,0,1\}$ on an interval $I=[0,T]$. If $\varepsilon=1$, suppose also $|I|\leq1$.
 For any B-admissible pairs $(q,r)$ and $(\bar q, \bar r)$,
 \begin{equation}\label{stri}
\Vert u\Vert_{L^q(I,L^r)}\leq c (\Vert u_0\Vert_{L^2}+ \Vert
h\Vert_{L^{\bar{q}'}(I,L^{\bar{r}'})}),
\end{equation}
where $\bar q'$ and $\bar r'$ are the conjugate exponents of $\bar
q$ and $\bar r$, i.e., $\frac1{\bar q}+\frac1{\bar q'}=\frac1{\bar
r}+\frac1{\bar r'}=1$. Besides, for any S-admissible pairs $(q,r)$
and $(a,b)$, and any $s\geq0$,
 \begin{equation}\label{sstri}
\Vert |\nabla|^{s}u\Vert_{L^q(I,L^r)}\leq c (\Vert
|\nabla|^{s-\frac2q}u_0\Vert_{L^2}+ \Vert
|\nabla|^{s-\frac2q-\frac2a}h\Vert_{L^{a'}(I,L^{b'})}).
\end{equation}
 From Sobolev embedding,  estimates \eqref{sstri} implies \eqref{stri}.
 We define two norms for convenience to study the  $\dot H^{s_c}$-critical equation \eqref{1.1} by
\begin{align}\label{Z(I)}
&\Vert u\Vert_{Z(I)}=\Vert
u\Vert_{L^{\frac{(N+4)(p-1)}{4}}(I,L^{\frac{(N+4)(p-1)}{4}})},\
\ \ \
 \Vert u\Vert_{Z'(I)}=\Vert
u\Vert_{L^{\frac{(N+4)(p-1)}{4p}}(I,L^{\frac{(N+4)(p-1)}{4p}})}.
\end{align}
Thus a direct consequence of \eqref{sstri} and the  Sobolev's inequality is that,
  if $u\in C(I,H^{-4}(\mathbb{R}^N))$ be a solution of \eqref{eq}
 with $u_0\in\dot H^2$ and $\nabla h\in L^{2}(I,L^{\frac{2N}{N+2}})$,
 then  $u\in C(I,\dot H^{2}(\mathbb{R}^N))$
 and for any B-admissible pairs $(q,r)$,
 \begin{equation}\label{319}
\Vert\Delta u\Vert_{L^q(I,L^r)}\leq c (\Vert \Delta u_0\Vert_{L^2}+
\Vert \nabla h\Vert_{L^{2}(I,L^{\frac{2N}{N+2}})}).
\end{equation}
A key feature of \eqref{319} is that the second derivative of $u$ is estimated using only
one derivative of the forcing term $h$.
Just as in \cite{Pa1}, the Strichartz estimates yield the following local well-posedness result.
\begin{proposition}\label{posedness}
Given any initial data  $u_{0}\in H^2$, any $p\in(1,\frac{N+4}{N-4})$ when $N\geq5$, and any $p>1$
when $N\leq4$, there exists $T>0$ and a unique solution  $u\in C([0,T],H^2)$ of \eqref{1.1} with initial data $u_0$.
The solution has conserved mass and energy. Besides, if $T^+$ is the maximal time of existence of $u$, then
$\lim_{t\rightarrow T^+}\|u(t)\|_{H^2}=+\infty$ when $T^+<\infty$. And the solution map $u_0\mapsto u$ is continuous in the
sense that for any $T\in(0,T^+)$, if $u_0^k\in H^2$ is a sequence converging in $H^2$ to $u_0$, and if
$u^k$ denotes the solution of \eqref{1.1} with initial data $u_0^k$, then $u^k$ is defined on $[0,T]$ for sufficiently large
$k$ and $u^k\rightarrow u$ in $C([0,T],H^2)$.

\end{proposition}

The author in \cite{robert taggart} (Theorem 1.4 there)  studied the inhomogeneous
Strichartz estimates for dispersive operators in the abstract
setting, which indeed  yield the counterpart ones for the fourth-order
Schr\"odinger operators by the same method used in the proof of Corollary 7.1 in \cite{robert taggart}.
More precisely, we say that a pair  $(q,r)$ is $\dot{H}^{s}$-biharmomic admissible and denote it by
 $(q,r)\in\Lambda_{s}$ if $0\leq s<2$ and
$$\frac{4}{q}+\frac{N}{r}=\frac{N}{2}-s,\ \ \ \frac{2N}{N-2s}\leq r<\frac{2N}{N-4}.$$
Correspondingly, we call the pair $(q',r')$
dual~$\dot{H}^{s}$-biharmomic  admissible, denoted by
$(q',r')\in\Lambda'_{s}$, if
 $(q,r)\in\Lambda_{-s}$ and $(q',r')$ is the conjugate exponent pair of $(q,r).$
In particular, $(q,r)\in\Lambda_0$ is just a B-admissible pair, which is always denoted by $(q,r)\in\Lambda_ B$.
Combining the results obtained by \cite{Pa1} and \cite{robert taggart}, we can infer
the following  inhomogeneous Strichartz estimates on  $I=[0,T]$:
\begin{align}\label{inhomo}
\left\|\int_0^te^{i(t-t^1)\Delta^2}f(\cdot,s)ds\right\|_{L^q(I;L^r)}\leq c\|f\|_{L^{\tilde q'}(I;L^{\tilde r'})},
\ \ \ \ \forall\ \  (q,r)\in\Lambda_{s},\ \ \forall\ \  (\tilde q,\tilde r)\in\Lambda_{-s}.
\end{align}
We also refer to \cite{foschi1,kato1,keel-tao,vilela} for more precise  discussion on the inhomogeneous Strichartz estimates.

Note that in the definition of $\|\cdot\|_{Z(I)}$ and
$\|\cdot\|_{Z'(I)}$, the pair
$(\frac{(N+4)(p-1)}{4},\frac{(N+4)(p-1)}{4})\in\Lambda_{s_c}$ and
$(\frac{(N+4)(p-1)}{4p},\frac{(N+4)(p-1)}{4p})\in\Lambda'_{s_c}$
with $s_c=\frac2N-\frac4{p-1}$ defined in the introduction. Since
for any $(q,r)\in\Lambda_{s_c}$, we can check that $(\frac qp,\frac
rp)\in\Lambda'_{s_c}$, then the inhomogeneous Strichartz estimates
\eqref{inhomo} combined with the H\"older inequality give that
\begin{align}\label{in-stri}
\left\|\int_0^te^{i(t-s)\Delta^2}|u|^{p-1}u(s)ds\right\|_{L^q(I;L^r)}\leq c\||u|^{p-1}u\|_{L^{
\frac qp}(I;L^{\frac rp})}\leq c\left\|u\right\|^p_{L^q(I;L^r)}.
\end{align}

As a consequence of the Strichartz estimates introduced above, we can obtain the following proposition.
\begin{proposition}\label{sd}
Assume $u_{0}\in H^2$, $t_0\in I$ an interval of $\mathbb R$. Then
there exists  $\delta_{sd}>0$ such that if
 $\|e^{it\Delta^2}u_{0}\|_{Z(I)}\leq \delta_{sd}, $ then there exists a unique solution
 $u\in C(I,H^2)$ of \eqref{1.1} with initial data $u_0$. This solution has conserved mass and  energy,
 and satisfies
\begin{align}\label{sd1}
 \|u\|_{Z(I)}\leq
2\delta_{sd},\ \ \ \|u\|_{L^\infty(I,H^2)}\leq c\|u_0\|_{H^2}.
\end{align}
\end{proposition}

\begin{proof}
 For
$\delta=\delta_{sd}$ and $M=c\|u_0\|_{H^2}$, we define a map as
$$\Phi(u)=
e^{i(t-t_0)\Delta^2}u_0+i\int_{t_0}^te^{i(t-s)\Delta^2}|u|^{p-1}u(s)ds,$$
and a set as
$$M_{M,\delta}=\{v\in C(I,H^2):\ \ \|v\|_{Z(I)}\leq2\delta,
\ \ \|v\|_{L^{2(p-1)}(I,L^{\frac{N(p-1)}{2}})}\leq2\delta,\ \ \  \|\Delta v\|_{L^\infty(I,L^2)}\leq2M
\}$$ equipped with the $Z(I)$ norm.
Then from the Strichartz estimates \eqref{319} and
\eqref{in-stri}, using the Sobolev embedding  and the H\"older inequalities,  we have for any $u\in M_{M,\delta}$,
\begin{align*}
\|\Phi(u)\|_{Z(I)}\leq\delta+c\|u\|_{Z(I)}^{p},\ \ \ \
\|\Phi(u)\|_{L^{2(p-1)}(I,L^{\frac{N(p-1)}{2}})}\leq\delta+c\|u\|_{Z(I)}^{p}
\end{align*}
and
\begin{align*}
\|\Delta\Phi(u)\|_{L^\infty(I,L^2)}&
\leq c\|\Delta u_0\|_2+c\|u\|_{L^{2(p-1)}(I,L^{\frac{N(p-1)}{2}})}^{p-1}\|\nabla u\|_{L^\infty(I,L^{\frac{2N}{N-2}})}\\
 &\leq c\|\Delta u_0\|_2+c\|u\|_{L^{2(p-1)}(I,L^{\frac{N(p-1)}{2}})}^{p-1}\|\Delta u\|_{L^\infty(I,L^2)}.
\end{align*}
Moreover, for any $u,v\in M_{M,\delta}$,
\begin{align*}
\|\Phi(u)-\Phi(v)\|_{Z(I)}\leq c(\|u\|_{Z(I)}^{p-1}+\|v\|_{Z(I)}^{p-1})\|u-v\|_{Z(I)}.
\end{align*}
From a standard argument, we can obtain that if $\delta$ is
sufficiently small, the map $u\mapsto \Phi(u)$ is a contraction map
on $M_{M,\delta}$. Thus,  the contraction mapping theorem  gives a
unique solution u in $M_{M,\delta}$ satisfying  \eqref{sd1}.
\end{proof}

From the small data theory (Proposition \ref{sd}) and using a similar argument
as in \cite{Pa1}, we can obtain the following  result of scattering, the proof of which is standard and we omit here.
\begin{proposition}\label{h2scattering}
Let $u(t)\in C(\mathbb R,H^2)$ be a solution of \eqref{1.1}. If
 $\|u\|_{Z(\mathbb R)}<\infty$,
then $u(t)$ scatters in $H^2$.
That is , there exists $\phi^{\pm}\in H^2$ such that
$\lim_{t\rightarrow\pm\infty}\|u(t)-e^{it\Delta}\phi^{\pm}\|_{H^2}=0.$
\end{proposition}

Now we show a useful perturbation lemma as follows.
\begin{lemma}\label{perturb}
 For any given $A$, there exist $\epsilon_0=\epsilon_0(A,N,p)$ and
$c=c(A)$ such that for any $\epsilon\leq\epsilon_0$, any interval $I
=(T_1,T_2)\subset \mathbb R$ and any
 $\tilde{u}=\tilde{u}(x,t)\in H^2$ satisfying
$$ i\tilde{u}_{t}+\Delta^2 \tilde{u}-|\tilde{u}|^{p-1}\tilde{u}=e,$$
if for some $(q,r)\in\Lambda_{-s_c}$,
 $$\|\tilde{u}\|_{Z(I)}\leq A,\ \   \|e\|_{L^{q'}(I;L^{r'})}\leq \epsilon$$
and $$\|e^{i(t-t_0)\Delta^2}(u(t_0)-\tilde{u}(t_0)\|_{Z(I)}\leq
\epsilon,$$ then the solution $u\in C(I;H^2)$ of \eqref{1.1}
satisfying
 $$\|u-\tilde u\|_{Z(I)}\leq c(A)\epsilon.$$
\end{lemma}

\begin{proof}
Let $w$ be defined by $u=\tilde u+w$. Then  $w$ solves the equation
\begin{align}\label{w}
i\partial_t w+\Delta^2w-|w+\tilde u|^{p-1}(w+\tilde u)+|\tilde u|^{p-1}\tilde u+e=0.
\end{align}
For any $t_0\in I$,  $I=(T_1,t_0]\cup[t_0,T_2)$.  We need only
consider on $I_+=[t_0,T_2)$, since the case on $I_-=(T_1,t_0]$ can be
considered similarly.  Since $\|\tilde{u}\|_{Z(I)}\leq A$, we can
partition $[t_0,T_2)$ into $N=N(A)$ intervals $I_j =[t_j,t_{j+1}]$
such that for each $j$, the quantity
$\|\tilde{u}\|_{Z(I_j)}\leq\delta$ is suitably small with $\delta$
to be chosen later. The integral equation of $w$ with initial time
$t_j$ is
\begin{align}\label{w2}
  w(t)=e^{i(t-t_j)\Delta^2}w(t_j)-i\int_{t_j}^te^{i(t-s)\Delta^2}[|w+\tilde u|^{p-1}(w+\tilde u)-|\tilde u|^{p-1}\tilde u-e](s)ds.
\end{align}
Using the inhomogeneous Strichartz estimates \eqref{in-stri} on $I_j$, we obtain
\begin{align*}
\|w\|_{Z(I_j)}&\leq\|e^{i(t-t_j)\Delta^2}w(t_j)\|_{Z(I_j)}
+c\||w+\tilde u|^{p-1}(w+\tilde u)+|\tilde u|^{p-1}\tilde u\|_{Z'(I_j)}+\|e\|_{L^{q'}(I;L^{r'})}\\
&\leq\|e^{i(t-t_j)\Delta^2}w(t_j)\|_{Z(I_j)}
+c\| \tilde u\|^{p-1}_{Z (I_j)}\|w\|_{Z (I_j)}+c\|w\|_{Z (I_j)}^p+\|e\|_{L^{q'}(I;L^{r'})}\\
&\leq\|e^{i(t-t_j)\Delta^2}w(t_j)\|_{Z(I_j)} +c\delta^{p-1}\|w\|_{Z
(I_j)}+c\|w\|_{Z (I_j)}^p+c\epsilon_0.
\end{align*}
If
\begin{align}\label{condition}
\delta\leq\left(\frac1{4c}\right)^{\frac1{p-1}},\ \ \ \
\|e^{i(t-t_j)\Delta^2}w(t_j)\|_{Z(I_j)}+c\epsilon_0\leq\frac12\left(\frac1{4c}\right)^{\frac1{p-1}},
\end{align}
then 
\begin{align*}
\|w\|_{Z(I_j)}\leq2\|e^{i(t-t_j)\Delta^2}w(t_j)\|_{Z(I_j)}+c\epsilon_0.
\end{align*}
Now take $t=t_{j+1}$ in \eqref{w2}, and apply $e^{i(t-t_{j+1})\Delta^2}$
to both sides to obtain
\begin{align}\label{w2}
 e^{i(t-t_{j+1})\Delta^2} w(t_{j+1})=e^{i(t-t_j)\Delta^2}w(t_j)-i\int_{t_j}^{t_{j+1}}e^{i(t-s)\Delta^2}[|w+\tilde u|^{p-1}(w+\tilde u)-|\tilde u|^{p-1}\tilde u-e](s)ds.
\end{align}
Since the Duhamel integral is confined to $I_j$, using the inhomogeneous Strichart'z estimates \eqref{in-stri}
and following a similar argument as above, we obtain that
\begin{align*}
\| e^{i(t-t_{j+1})\Delta^2}
w(t_{j+1})\|_{Z(I_+)}&\leq\|e^{i(t-t_j)\Delta^2}w(t_j)\|_{Z(I_+)}
+c\delta^{p-1}\|w\|_{Z (I_j)}+c\|w\|_{Z (I_j)}^p+c\epsilon_0\\
&\leq2\|e^{i(t-t_j)\Delta^2}w(t_j)\|_{Z(I_+)}+c\epsilon_0.
\end{align*}
Iterating beginning with $j=0$, we obtain
\begin{align*}
\| e^{i(t-t_{j})\Delta^2}
w(t_{j})\|_{Z(I_+)}\leq2^{j}\|e^{i(t-t_0)\Delta^2}w(t_0)\|_{Z(I_+
)}+(2^{j}-1)c\epsilon_0 \leq2^{j+2}c\epsilon_0.
\end{align*}
To accommodate the conditions \eqref{condition} for all intervals $I_j$ with $0\leq j\leq N-1$, we require
\begin{align}\label{condition'}
2^{N+2}c\epsilon_0\leq\left(\frac1{4c}\right)^{\frac1{p-1}}.
\end{align}
Finally,$$\|w\|_{Z(I_+)}\leq
\sum_{j=0}^{N-1}2^{j+2}c\epsilon_0+cN\epsilon_0\leq
c(N)\epsilon_0,$$ which implies $\|w\|_{Z(I_+)}\leq c(A)\epsilon_0$
since $N=N(A)$, concluding the proof.

\end{proof}

\section{Variational Structure}
Following the idea from  \cite{W1}, we study  the
variational structure of the ground state of the   elliptic
equation
\eqref{Q}
by seeking   the best constant of the Gagliardo-Nirenberg inequality
\begin{align}\label{gn}
\|u\|_{p+1}^{p+1}\leq C_{GN} \|u\|_2^{p+1-\frac{N(p-1)}{4}}\|\Delta
u\|_{2}^{\frac{N(p-1)}{4}}.
\end{align}

Formally,  if $W$ is the minimizer of the variational problem
\begin{align}\label{J}
J=\inf\{J(u):u\in H^2\} \ \ \ \ with\ \  \ \ J(u)=\frac{\|u\|_2^{p+1-\frac{N(p-1)}{4}}\|\Delta u\|_{2}^{\frac{N(p-1)}{4}}}{\| u\|_{p+1}^{p+1}},
\end{align}
then we compute straightforward to get that
 $W$ satisfies the equation
\begin{align*}
&\|W\|_2^{p-1-\frac{N(p-1)}{4}}\|\Delta W\|_2^{\frac{N(p-1)}{4}}(1+\frac{(4-N)(p-1)}{8})W\\
&+\|W\|_2^{p+1-\frac{N(p-1)}{4}}\|\Delta W\|_2^{\frac{N(p-1)}{4}-2}\frac{N(p-1)}{8}\Delta^2 W-J^{\frac{p+1}2}|W|^{p-1}W=0.
\end{align*}
If we  set $W(x)=aQ(x)$, where $a,b$ satisfies
$$\frac{N(p-1)}8b^4(2-s_c)(1+\frac{(4-N)(p-1)}{8})^{-1}=1$$ and
$$J^{\frac{p+1}2}a^{p-1}(2-s_c)(1+\frac{(4-N)(p-1)}{8})^{-1}=1,$$ then
$Q(x)=a^{-1}W(b^{-1}x)$ solves the equation \eqref{Q}, and also
attains the variational problem \eqref{J} with
$J=J(Q)=\frac{\|Q\|_2^{p+1-\frac{N(p-1)}{4}}\|\Delta
Q\|_{2}^{\frac{N(p-1)}{4}}}{\|Q\|_{p+1}^{p+1}}$ (noting that
$J(Q)=J(W)$ is invariant under the scaling $Q(x)=a^{-1}W(b^{-1}x)$).

The existence of the ground state solution of \eqref{Q} can be shown
by the same method as used in \cite{zhu,lenzmann}, so we omit here. Moreover, the
Pohozeav identity
$$(2-\frac N2)\|\Delta Q\|_2^2-(2-s_c)\frac N2\|Q\|_2^2+\frac{N}{p+1}\|Q\|_{p+1}^{p+1}=0,$$
which can be obtained by multiplying the equation \eqref{Q} by
$x\cdot\nabla Q$, combined with the identity $\|\Delta
Q\|_2^2+(2-s_c)\|Q\|_2^2-\|Q\|_{p+1}^{p+1}=0,$ obtained by
multiplying the equation \eqref{Q} by $Q$, implies  immediately that
\begin{align}\label{id}
\|\Delta Q\|_2^2=\frac{N(p-1)}{4(p+1)}\|Q\|_{p+1}^{p+1},\ \
\|Q\|_2^2=\frac{p-1}{2(p+1)}\|Q\|_{p+1}^{p+1},\ \
E(Q)=\frac{N(p-1)-8}{8(p+1)}\|Q\|_{p+1}^{p+1}.
\end{align}
Then we have \begin{align}\label{cgn} C_{GN}=\frac
1J=\frac{4(p+1)}{N(p-1)}\frac1{\|Q\|_2^{p+1-\frac{N(p-1)}{4}}\|\Delta
Q\|_{2}^{\frac{N(p-1)}{4}-2} }.\end{align}

\section{Global versus blow-up and dichotomy}

\begin{theorem}\label{dichotomy}
 Let $u_{0}\in H^{2}$ and  $I=(T_{-},T_{+})$  be the
maximal time interval of existence of ~$u(t)$~ solving ~\eqref{1.1}.~
Suppose that
\begin{equation}\label{2.1'}
M(u)^{\frac{2-s_c}{s_c}}E(u)<M(Q)^{\frac{2-s_c}{s_c}}E(Q).
\end{equation}
If ~\eqref{2.1'}~holds and
\begin{equation}\label{2.2'}
\|u_{0}\|^{\frac{2-s_c}{s_c}}_{2}\|\Delta u_{0}\|_{2}<\|Q\|^{\frac{2-s_c}{s_c}}_{2}\|\Delta Q\|_{2},
\end{equation}
then ~$I=(-\infty,+\infty)$,~ i.e., the solution exists globally in
time, and for all time ~$t\in \mathbb{R},$~
\begin{equation}\label{2.3'}
\|u(t)\|^{\frac{2-s_c}{s_c}}_{2}\|\Delta u(t)\|_{2}<\|Q\|^{\frac{2-s_c}{s_c}}_{2}\|\Delta Q\|_{2}.
\end{equation}
If ~\eqref{2.1'}~holds and
\begin{equation}\label{2.4'}
\|u_{0}\|^{\frac{2-s_c}{s_c}}_{2}\|\Delta u_{0}\|_{2}>\|Q\|^{\frac{2-s_c}{s_c}}_{2}\|\Delta Q\|_{2},
\end{equation}
then for~$t\in I,$~
\begin{equation}\label{2.5'}
\|u(t)\|^{\frac{2-s_c}{s_c}}_{2}\|\Delta u(t)\|_{2}>\|Q\|^{\frac{2-s_c}{s_c}}_{2}\|\Delta Q\|_{2}.
\end{equation}

\end{theorem}

\begin{proof}
Multiplying the definition of energy by $M(u)^{\frac{2-s_c}{s_c}}$ and using \eqref{gn},
we have
\begin{align*}
M(u)^{\frac{2-s_c}{s_c}}E(u)=&\frac12 \|u(t)\|^{\frac{2(2-s_c)}{s_c}}_{2}\|\Delta u(t)\|^2_{2}
-\frac1{p+1}\|u\|_{p+1}^{p+1}\|u\|_2^{\frac{2(2-s_c)}{s_c}}\\
\geq& \frac12 (\|u(t)\|^{\frac{2-s_c}{s_c}}_{2}\|\Delta u(t)\|_{2})^2
-\frac{C_{GN}}{p+1}(\|u(t)\|^{\frac{2-s_c}{s_c}}_{2}\|\Delta u(t)\|_{2})^{\frac{N(p-1)}{4}}.
\end{align*}
Define $f(x)=\frac12x^2-\frac1{p+1}C_{GN}x^{\frac{N(p+1)}{4}}$. Then
$f'(x)=x\left(1-C_{GN}\frac{N(p-1)}{4(p+1)}x^{\frac{N(p+1)-8}{4}}\right)$, and
thus, $f'(x)=0$ when $x_0=0$ and
$x_1=\|Q\|^{\frac{2-s_c}{s_c}}_{2}\|\Delta Q\|_{2}$. The graph of
$f$ has a local minimum at $x_0$ and a local maximum at $x_1$. The
condition \eqref{2.1'} and \eqref{id} imply that
$M(u_0)^{\frac{2-s_c}{s_c}}E(u_0)<f(x_1)=M(Q)^{\frac{2-s_c}{s_c}}E(Q)$.
This combined with energy conservation gives that
\begin{align}\label{1}
f(\|u(t)\|^{\frac{2-s_c}{s_c}}_{2}\|\Delta u(t)\|_{2})\leq M(u(t))^{\frac{2-s_c}{s_c}}E(u(t))=M(u_0)^{\frac{2-s_c}{s_c}}E(u_0)
<f(x_1).
\end{align}
If initially $\|u_{0}\|^{\frac{2-s_c}{s_c}}_{2}\|\Delta
u_{0}\|_{2}<x_1$, then by  \eqref{1} and the continuity of $\|\Delta
u(t)\|_2$ in $t$, we have \eqref{2.3'} for all time $t\in I$. In
particular, the $H^2$-norm of the solution $u$ is bounded, which, by Proposition \ref{posedness},
proves the global existence in this case.
If initially $\|u_{0}\|^{\frac{2-s_c}{s_c}}_{2}\|\Delta u_{0}\|_{2}>x_1$, then by  \eqref{1}
and the continuity of $\|\Delta u(t)\|_2$ in $t$, we have \eqref{2.5'} for all time $t\in I$.

From the argument above, we
  can refine this analysis to obtain the following.  If the
condition \eqref{2.2'} holds, then there exists $\delta>0$ such that
$M(u)^{\frac{2-s_c}{s_c}}E(u)<(1-\delta)M(Q)^{\frac{2-s_c}{s_c}}E(Q)$,
and thus there exists $\delta_0=\delta_0(\delta)$ such that
$\|u(t)\|^{\frac{2-s_c}{s_c}}_{2}\|\Delta u(t)\|_{2}<
(1-\delta_0)\|Q\|^{\frac{2-s_c}{s_c}}_{2}\|\Delta Q\|_{2}.$
\end{proof}

The next two lemmas provide some additional estimates  under the
hypotheses \eqref{2.1'} and \eqref{2.2'} in Theorem \ref{dichotomy}.
These lemmas
 will be needed in the proof of Theorem \ref{th1} through a virial-type estimate, which will be established  in the last
two sections.

\begin{lemma}\label{lower bound}
Let $u_0\in H^2$ satisfy \eqref{2.1'} and \eqref{2.2'}. Furthermore,
take $\delta>0$ such that
$M(u_0)^{\frac{2-s_c}{s_c}}E(u_0)<(1-\delta)M(Q)^{\frac{2-s_c}{s_c}}E(Q)$.
If $u$ is a solution of problem \eqref{1.1} with initial data $u_0$,
then there exists $C_\delta>0$ such that for all $t\in\mathbb R$,
\begin{align}\label{4.7}
\|\Delta u\|_2^2-\frac{N(p-1)}{4(p+1)}\|u\|_{p+1}^{p+1}\geq
C_\delta\|\Delta u\|_2^2.
\end{align}
\end{lemma}
\begin{proof}
By the analysis in the proof of Theorem \ref{dichotomy}, there exists
$\delta_0=\delta_0(\delta)>0$ such that for all $t\in\mathbb R$,

\begin{align}\label{low}
\|u(t)\|^{\frac{2-s_c}{s_c}}_{2}\|\Delta u(t)\|_{2}<
(1-\delta_0)\|Q\|^{\frac{2-s_c}{s_c}}_{2}\|\Delta Q\|_{2}.
\end{align} Let
$$h(t)=\frac1{\|Q\|^{\frac{2(2-s_c)}{s_c}}_{2}\|\Delta Q\|^2_{2}}(
\|u(t)\|^{\frac{2(2-s_c)}{s_c}}_{2}\|\Delta
u(t)\|^2_{2}-\frac{N(p-1)}{4(p+1)}\|u\|_{p+1}^{p+1}\|u(t)\|^{\frac{2(2-s_c)}{s_c}}_{2})$$
and set $g(y)=y^2-y^{\frac{N(p-1)}4}$. By Gagliardo-Nirenberg
estimate \eqref{gn} with sharp constant $C_{GN}$ \eqref{cgn}, we can obtain
$$h(t)\geq g\left(\frac{\|u(t)\|^{\frac{2-s_c}{s_c}}_{2}\|\Delta u(t)\|_{2}}{\|Q\|^{\frac{2-s_c}{s_c}}_{2}\|\Delta Q\|_{2}}\right).$$
By \eqref{low}, we restrict our attention to $0\leq y\leq1-\delta_0$.
The elementary argument gives a constant $C_\delta$ such that $g(y)\geq
C_\delta y^2$ if $0\leq y\leq1-\delta_0$. This indeed implies \eqref{4.7}.
\end{proof}

\begin{lemma}\label{4.5}
(Comparability of the kinetic energy and the total energy)
Let $u_0\in H^2$ satisfy \eqref{2.1'} and \eqref{2.2'}. Then
$$\frac{N(p-1)-8}{2N(p-1)}\|\Delta u(t)\|_{2}^2\leq E(u)\leq\frac{1}{2}\|\Delta u(t)\|_{2}^2.$$
\end{lemma}

\begin{proof}
The expression of $E(u)$ gives the second inequality immediately.
The first one can be obtained from
\begin{align*}
& \frac{1}{2}\|\Delta u\|_{2}^2-\frac{1}{p+1}\|u\|_{p+1}^{p+1}\geq\frac{1}{2}\|\Delta u\|_{2}^2
(1-\frac{2C_{GN}}{p+1}\|\Delta u\|^{\frac{N(p+1)}4-2}_{2}\|  u\|^{p+1-\frac{N(p+1)}4}_{2})\\
=&\frac{1}{2}\|\Delta u\|_{2}^2
\left(1-\frac{8}{N(p-1)}\left(\frac{\|u(t)\|^{\frac{2-s_c}{s_c}}_{2}\|\Delta
u(t)\|_{2}} {\|Q\|^{\frac{2-s_c}{s_c}}_{2}\|\Delta
Q\|_{2}}\right)^{\frac{N(p-1)-8}{4}}\right)
\geq\frac{N(p-1)-8}{2N(p-1)}\|\Delta u\|_{L^2}^2,
\end{align*}
where we have used \eqref{gn} and \eqref{id}.

\end{proof}

To establish the scattering theory,
  we need the following result.

\begin{proposition}\label{wave operator}
(Existence of wave operators) Suppose $ \psi^+\in H^2$ and
\begin{equation}\label{4.11}
\frac{1}{2}\|\psi^+\|^{\frac{2(2-s_c)}{s_c}}\|\Delta\psi^+\|_2^2<E(Q)M(Q)^{\frac{2-s_c}{s_c}}.
\end{equation}
Then there exists $v_0\in H^2$ such that the solution $v$ of
\eqref{1.1} with initial data $v_0$ satisfies
$$\|\Delta v(t)\|_2\|v_0\|^{\frac{2-s_c}{s_c}}_2<\|\Delta Q\|_2\|Q\|^{\frac{2-s_c}{s_c}}_2,\ \
M(v)=\|\psi^+\|_2^2,\ \  E(v)=\frac{1}{2}\|\Delta\psi^+\|_2^2,$$ and
 $\lim_{t\rightarrow+\infty}\|v(t)-e^{it\Delta}\psi^+\|_{H^2}=0.$
Moreover, if
$\|e^{it\Delta}\psi^+\|_{Z([0,\infty))}\leq\delta_{sd},$ then
$$\|v\|_{Z([0,\infty))}\leq c\|e^{it\Delta}\psi^+\|_{Z([0,\infty))},\ \
         \|D^{s_c}v\|_{2}\leq c\|\psi^+\|_{H^2}.$$
         A similar result holds for the case $t\rightarrow-\infty$.
\end{proposition}

\begin{proof}
Similar to the proof of the small data scattering theory Proposition \ref{sd}, we can solve the integral equation
\begin{equation}\label{4.12}
v(t)=e^{it\Delta^2}\psi^++i\int_t^\infty
e^{i(t-s)\Delta^2}|v|^{p-1}v(s)ds
\end{equation}
 for $t\geq T$ with $T$ large.

 In fact, there exists some large $T$ such that
$\|e^{it\Delta}\psi^+\|_{Z([T,\infty))}\leq \delta_{sd},$
where $\delta_{sd}$ is defined by Proposition \ref{sd}.
Then, the same arguments as used in Proposition \ref{sd}  give a solution $v\in C([T,\infty),H^2)$ of \eqref{4.12}.
Moreover, we also have
$
\|v\|_{Z([T,\infty))}\leq2\delta_{sd},
$ $\|v\|_{L^{2(p-1)}([T,\infty),L^{\frac{N(p-1)}{2}})}\leq2\delta_{sd}$,
and $\|\Delta v\|_{L^{\infty}([T,\infty);L^2)} \leq c\|\Delta
v_0\|_2$.
Thus from
\begin{align*}
\|\Delta(v-e^{it\Delta^2}\psi^+)\|_{L^{\infty}([T,\infty);L^2)}
 \leq c\|v\|_{L^{2(p-1)}([T,\infty),L^{\frac{N(p-1)}{2}})}^{p-1}
 \|\Delta v\|_{L^\infty([T,\infty),L^2)},
\end{align*}
we get that $$\|\Delta(v-e^{it\Delta^2}\psi^+)\|_{L^{\infty}([T,\infty);L^2)} \rightarrow 0
\ \  as ~T\rightarrow \infty,$$
which implies $v(t)-e^{it\Delta^2}\psi^+\rightarrow0 $ in $H^2$ as $t\rightarrow+\infty$.
Thus $M(v)=\|\psi^{+}\|_{2}^2.$

Since $e^{it\Delta^2}\psi^+\rightarrow0$ in $L^r$ as $t\rightarrow+\infty$ for any $r\in(2,\frac{2N}{N-4}),$ we
get easily that $\|e^{it\Delta^2}\psi^+\|_{p+1}\rightarrow0$. This together with the fact that
 $\|\Delta e^{it\Delta^2}\psi^+\|_2$ is conserved implies
 $$E(v)=\lim_{t\rightarrow\infty}\left(\frac{1}{2}\|\Delta e^{it\Delta^2}\psi^+\|_2^2
 -\frac{1}{p+1}\|e^{it\Delta^2}\psi^+\|_{p+1}^{p+1}\right)=\frac{1}{2}\|\Delta \psi^+\|_2^2.$$
 In view of \eqref{4.11} we immediately obtain $M(v)^{\frac{ 2-s_c }{s_c}}E(v)<E(Q)M(Q)^{\frac{ 2-s_c }{s_c}}.$
  Note that
\begin{align*}
&\lim_{t\rightarrow\infty}\|v(t)\|_2^{\frac{2(2-s_c)}{s_c}}\|\Delta v(t)\|_2^2
=\lim_{t\rightarrow\infty}\|e^{it\Delta}\psi^+\|_2^{\frac{2(2-s_c)}{s_c}}\|\Delta e^{it\Delta}\psi^+\|_2^{\frac{2(2-s_c)}{s_c}}\\
=&\|\psi^+\|_2^{\frac{2(2-s_c)}{s_c}}\|\Delta
\psi^+\|_2^2<2E(Q)M(Q)^{\frac{2-s_c}{s_c}}=\frac{N(p-1)-8}{N(p-1)}\|Q\|_2
^{\frac{2(2-s_c)}{s_c}}\|\Delta Q\|_2^2,
\end{align*}
where we have used \eqref{4.11} and \eqref{id} in the last two steps.
Thus, due to Theorem \ref{dichotomy}, we can evolve $v(t)$ from $T$ back to the initial time $0$,
concluding our proof.

\end{proof}

\section{Existence and  compactness of a critical element}

\begin{definition}
We say that $SC(u_0)$ holds if for $u_0\in H^2$ satisfying
$\|u_{0}\|^{\frac{2-s_c}{s_c}}_{2}\|\Delta
u_{0}\|_{2}<\|Q\|^{\frac{2-s_c}{s_c}}_{2}\|\Delta Q\|_{2}$ and
$E(u_{0})M(u_{0})^{\frac{2-s_c}{s_c}}<E(Q)M(Q)^{\frac{2-s_c}{s_c}}$,
 the corresponding solution $u$  of \eqref{1.1} with the maximal
interval of existence $I=(-\infty,+\infty)$ satisfies
\begin{equation}\label{Sbound}
\|u\|_{Z(\mathbb R)}<\infty.
\end{equation}
\end{definition}

We first claim  that  there exists $\delta>0$ such that if
$E(u)M(u)^{\frac{2-s_c}{s_c}}<\delta$ and
$\|u_{0}\|^{\frac{2-s_c}{s_c}}_{2}\|\Delta
u_{0}\|_{2}<\|Q\|^{\frac{2-s_c}{s_c}}_{2}\|\Delta Q\|_{2},$ then
\eqref{Sbound} holds. In fact, by the definition of norm $\|\cdot\|_{Z(I)}$, the B-Strichartz estimate \eqref{stri}
and Lemma \ref{4.5}, we have
$$\|e^{it\Delta^2}u_0\|_{Z(\mathbb R)}^{\frac2{s_c}}\leq c
\|u_0\|_{\dot{H}^{s_c}}^{\frac2{s_c}}\leq c\|u_{0}\|^{\frac{2(2-s_c)}{s_c}}_{2}\|\Delta
u_{0}\|_{2}^2 \leq
\frac{2N(p-1)c}{N(p-1)-8}E(u)M(u)^{\frac{2-s_c}{s_c}}.$$
So if $E(u)M(u)^{\frac{2-s_c}{s_c}}<\frac{N(p-1)-8}{2N(p-1)c}\delta_{sd}^{\frac2{s_c}}$,
 we get
that  $\|e^{it\Delta^2}u_0\|_{Z(\mathbb R)}\leq\delta_{sd}$.
Then from Proposition \ref{sd},  we get that $SC(u_0)$ holds, and the claim holds
for $\delta=\frac{N(p-1)-8}{2N(p-1)c}\delta_{sd}^{\frac2{s_c}}.$ Now
for each $\delta$, we define the set $S_\delta$ to be the collection
of all such initial data in $H^2$ :
$$S_\delta=\{u_0\in H^2:\ \  E(u)M(u)^{\frac{2-s_c}{s_c}}<\delta \ \  and \ \
\|u_{0}\|^{\frac{2-s_c}{s_c}}_{2}\|\Delta u_{0}\|_{2}<\|Q\|^{\frac{2-s_c}{s_c}}_{2}\|\Delta Q\|_{2} \}.$$
We also define $(M^{\frac{2-s_c}{s_c}}E)_c=\sup\{\delta:\ \ u_0\in S_\delta\Rightarrow SC(u_0)\ \  holds \}.$
If $(M^{\frac{2-s_c}{s_c}}E)_c=M(Q)^{\frac{2-s_c}{s_c}}E(Q)$, then we are done. Thus we assume now
\begin{equation}\label{me}
(M^{\frac{2-s_c}{s_c}}E)_c<M(Q)^{\frac{2-s_c}{s_c}}E(Q).
\end{equation}

\begin{remark}\label{r5}
By the definition of $(M^{\frac{2-s_c}{s_c}}E)_c$,  we can find a
sequence of solutions $u_n$ of \eqref{1.1} with  initial data
$u_{n,0} \in H^2$, which we rescale  to satisfy  $\|u_{n,0}\|_{2}=1$,  such that
$\|\Delta u_{n,0}\|_{2}<\|Q\|^{\frac{2-s_c}{s_c}}_{2}\|\Delta
Q\|_{2}$  and $E(u_n)\downarrow (M^{\frac{2-s_c}{s_c}}E)_c$ as
$n\rightarrow \infty,$ and  $SC(u_{n,0})$ does not hold for any $n$.
\end{remark}

Our goal in this section is to show the existence of an $H^2$
solution $u_c$ of \eqref{1.1} with the initial data $u_{c,0}$ such that
$\| u_{c,0}\|^{\frac{2-s_c}{s_c}}_{2}\|\Delta
u_{c,0}\|_{2}<\|Q\|^{\frac{2-s_c}{s_c}}_{2}\|\Delta Q\|_{2}$,
$M(u_c)^{\frac{2-s_c}{s_c}}E(u_c)= (M^{\frac{2-s_c}{s_c}}E)_c$ and
  $SC(u_{c,0})$ does not hold. Moreover, we  show that if
  $\|u_c\|_{Z([0,+\infty))}=\infty$, then
$K=\{u_c(x,t)|0\leq t<\infty\}$ is precompact in  $H^2$, and a
corresponding conclusion is reached if
$\|u_c\|_{Z((-\infty,0])}=\infty$.

Prior to fulfilling  our main task, we  first establish a
profile decomposition lemma using the concentration compactness principle  in the spirit  of
Keraani \cite{keraani} and Merle \cite{km}. We also refer to  \cite{radial}
for a similar result shown  for the 3D cubic Schr\"{o}dinger equation and to \cite{J-P-S} for
the linear profile decomposition for the
one-dimensional fourth-order Schr\"oinger equation.

\begin{lemma}\label{lpd}
(Profile decomposition). Let $\phi_{n}(x)$ be a radial uniformly
bounded sequence in $H^{2}$. Then for each $M$ there exists a
subsequence of $\phi_{n}$, which is denoted by itself, such that the
following statements hold. \\
(1) For each $1\leq j\leq M$, there exists
(fixed in n) a radial profile $\psi^{j}(x)$ in $H^2$ and
 a sequence (in $n$) of time
shifts $t_{n}^{j}$, and there exists a sequence (in $n$) of
remainders $W_{n}^{M}(x)$ in $H^2$  such that
$$\phi_{n}(x)=\sum_{j=1}^{M}e^{-it_{n}^{j}\Delta^2}\psi^{j}(x)+W_{n}^{M}(x).$$
(2) The time  sequences have a pairwise divergence property, i.e.,  for $1\leq j\neq k\leq M$,
\begin{equation}\label{divergence}
\lim_{n\rightarrow+\infty}
|t_{n}^{j}-t_{n}^{k}|=+\infty.
\end{equation}
(3) The remainder sequence has the following asymptotic smallness
property:
\begin{equation}\label{remainder}
\lim_{M\rightarrow+\infty}[\lim_{n\rightarrow+\infty}\|e^{it\Delta^2}W_{n}^{M}\|_{Z(\mathbb
R)}]=0.
\end{equation}
(4) For each fixed $M$ and any $0\leq s\leq2$, we have the asymptotic
Pythagorean expansion as follows
\begin{equation}\label{hsexpansion}
\|\phi_{n}\|_{\dot{H}^{s}}^{2}
=\sum_{j=1}^{M}\|\psi^{j}\|_{\dot{H}^{s}}^{2}+\|W_{n}^{M}\|_{\dot{H}^{s}}^{2}+o_{n}(1),
\end{equation}
where $o_{n}(1)\rightarrow0$ as $n\rightarrow+\infty$.
\end{lemma}

\begin{proof}
Let $c_1$ be such that $\|\phi_{n}\|_{H^2}\leq c_1$. By the definition of the norm $\|\cdot\|_{Z(I)}$, there
holds the interpolation inequality
$$\|v\|_{Z(\mathbb R)}\leq\|v\|^{1-\theta}_{L^q(\mathbb R;L^{r})}
\|v\|^\theta_{L^\infty(\mathbb R;L^{\frac{2N}{N-2s_c}})}$$ with some
$(q,r)\in\Lambda_{s_c}$ and some $\theta\in(0,1)$. This combined with the
Strichartz estimates gives
$$\|e^{it\Delta}W_{n}^{M}\|_{Z(\mathbb R)}\leq c\|W_{n}^{M}\|_{\dot H^{s_c}}^{1-\theta}\|e^{it\Delta}W_{n}^{M}\|^{\theta}_
{L^\infty(\mathbb R;L^{\frac{2N}{N-2s_c}})}.$$ Since
$\|W_{n}^{M}\|_{\dot H^{s_c}}\leq c_1$, it suffices to show that
\begin{equation}\label{remainder'}
\lim_{M\rightarrow+\infty}[\lim_{n\rightarrow+\infty}\|e^{it\Delta^2}W_{n}^{M}\|_
{L^\infty(\mathbb R;L^{\frac{2N}{N-2s_c}})}]=0.
\end{equation}

Let $A_1=\limsup_{n\rightarrow\infty}\|e^{it\Delta^2}\phi_n\|_
{L^\infty(\mathbb R;L^{\frac{2N}{N-2s_c}})}$. If $A_1=0$, the proof is
complete with $\psi^j=0$ for all $1\leq j\leq M$. Suppose $A_1>0$.
Passing to  a subsequence, we may assume that
$\lim_{n\rightarrow\infty}\|e^{it\Delta^2}\phi_n\|_
{L^\infty(\mathbb R;L^{\frac{2N}{N-2s_c}})}=A_1$. We will show that there is a
time sequence $t_n^1$ and a profile $\psi^1\in H^2$ such that
$e^{it_n^1\Delta^2}\phi_n\rightharpoonup\psi^1$ and
\begin{align}\label{5.7'}
\|\psi^1\|_{\dot H^{s_c}}\geq KA_1^{\frac{N}{2s_c}+\frac{N-2s_c}{\min\{2s_c,4-2s_c\}}}
\left(\frac1{c_1}\right)^{\frac{N}{2s_c}-1+\frac{N-2s_c}{\min\{2s_c,4-2s_c\}}}.
\end{align}
For $r>1$ to be chosen, let $\chi(x)$ be a radial Schwartz function such that $\hat{\chi}(\xi)=1$
for $\frac1r\leq|\xi|\leq r$ and $\hat{\chi}(\xi)$ is supported in $\frac1{2r}\leq|\xi|\leq 2r$.

Since the operator $e^{it\Delta^2}$ is an isometry  on $\dot H^{s_c}$, then
by the  Sobolev embedding,
\begin{align*}
&\|e^{it\Delta^2}\phi_n-\chi
e^{it\Delta^2}\phi_n\|^2_{L^\infty(\mathbb R;L^{\frac{2N}{N-2s_c}})}
\leq\sup_t\|e^{it\Delta^2}\phi_n-\chi e^{it\Delta^2}\phi_n\|^2_{\dot H^{s_c}}\\
 \leq&\sup_t\int
|\xi|^{2s_c}(1-\hat{\chi}(\xi))^2|\hat{\phi_n}(\xi)|^2d\xi
 \leq\int_{|\xi|\leq\frac1r}
|\xi|^{2s_c}|\hat{\phi_n}(\xi)|^2d\xi+
\int_{|\xi|\geq r}
|\xi|^{2s_c}|\hat{\phi_n}(\xi)|^2d\xi\\
 \leq&\frac1{r^{2s_c}}\|\phi_n\|_2+\frac1{r^{4-2s_c}}\|\phi_n\|_{\dot H^2}^2
 \leq(\frac1{r^{2s_c}}+\frac1{r^{4-2s_c}})c_1^2.
\end{align*}
Take $r$ sufficiently large such that
 $(\frac1{r^{2s_c}}+\frac1{r^{4-2s_c}})c_1^2=\frac{
A_1^2}4\epsilon_0$ with some $0<\epsilon_0<1$. (This implies that
$\frac1r\geq c\left(\frac{A_1^2}{c_1^2}
\right)^{\frac{1}{min\{2s_c,4-2s_c\}}}$.) Then, for $n$ large, we
have $\|\chi\ast
e^{it\Delta^2}\phi_n\|_{L^\infty(\mathbb R;L^{\frac{2N}{N-2s_c}})}\geq\frac12A_1$.
Thus from
\begin{align*}
\|\chi\ast
e^{it\Delta^2}\phi_n\|^{\frac{2N}{N-2s_c}}_{L^\infty(\mathbb R;L^{\frac{2N}{N-2s_c}})}
\leq \|\chi\ast e^{it\Delta^2}\phi_n\|^{2}_{L^\infty(\mathbb R;L^{2})}
\|\chi\ast
e^{it\Delta^2}\phi_n\|^{\frac{4s_c}{N-s_c}}_{L^\infty(\mathbb R;L^{\infty})}
\leq \|\phi_n\|^{2}_{2} \|\chi\ast
e^{it\Delta^2}\phi_n\|^{\frac{4s_c}{N-s_c}}_{L^\infty(\mathbb R;L^{\infty})},
\end{align*}
we get that
$\|\chi\ast e^{it\Delta^2}\phi_n\|_{L^\infty(\mathbb R;L^{\infty})}
\geq\left(\frac{A_1}2\right)^{\frac{N}{2s_c}}\left(\frac1{c^2_1}\right)^{\frac{N-s_c}{4s_c}}
$. Since $\phi_n$ are radial functions, so are $\chi\ast e^{it\Delta^2}\phi_n$.
By the radial Gagliardo-Nirenberg inequality, we obtain that
$$\|\chi\ast e^{it\Delta^2}\phi_n\|_{L^\infty(\mathbb R;L^{\infty}(|x|\geq R))}
\leq\frac1R\|\chi\ast e^{it\Delta^2}\phi_n\|_{L^\infty(\mathbb R;L^{2})}^{\frac12}
\|\nabla\chi\ast e^{it\Delta^2}\phi_n\|_{L^\infty(\mathbb R;L^{2})}^{\frac12}
\leq\frac{c_1}R.
$$
Therefore, by selecting $R$ large enough, $\|\chi\ast
e^{it\Delta^2}\phi_n\|_{L^\infty(\mathbb R;L^{\infty}(|x|\leq R))}
\geq\frac12\left(\frac{A_1}2\right)^{\frac{N}{2s_c}}\left(\frac1{c^2_1}\right)^{\frac{N-s_c}{4s_c}}$.
Let $t_n^1$ and $x_n^1$ with $|x_n^1|\leq R$ be the  sequences such
that for each $n$, $|\chi\ast e^{it_n^1\Delta^2}\phi_n(x_n^1)|
\geq\frac14\left(\frac{A_1}2\right)^{\frac{N}{2s_c}}\left(\frac1{c^2_1}\right)^{\frac{N-s_c}{4s_c}}$,
or $$\left|\int\chi(x_n^1-y) e^{it_n^1\Delta^2}\phi_n(y)dy\right|
\geq\frac14\left(\frac{A_1}2\right)^{\frac{N}{2s_c}}\left(\frac1{c^2_1}\right)^{\frac{N-s_c}{4s_c}}.$$
Passing to a subsequence such that $x_n^1\rightarrow x^1$, which is
possible because $|x_n^1|\leq R$, we obtain that
$$\left|\int\chi(x^1-y) e^{it_n^1\Delta^2}\phi_n(y)dy\right|
\geq\frac18\left(\frac{A_1}2\right)^{\frac{N}{2s_c}}\left(\frac1{c^2_1}\right)^{\frac{N-s_c}{4s_c}}.$$
Since the sequence $e^{it_n^1\Delta^2}\phi_n$ is uniformly bounded in $H^2$, then we can find a radial function  $\psi^1\in H^2$
such that, up to a subsequence, $e^{it_n^1\Delta^2}\phi_n\rightharpoonup\psi^1$
weakly in $H^2$ with $\|\psi^1\|_{H^2}\leq c_1$.
Thus,
$$
\left|\int\chi(x^1-y) \psi^1(y)dy\right|
\geq\frac18\left(\frac{A_1}2\right)^{\frac{N}{2s_c}}\left(\frac1{c^2_1}\right)^{\frac{N-s_c}{4s_c}}.$$
By Plancherel and Cauchy-Schwarz inequalities, $\|\chi\|_{\dot
H^{-s_c}}\|\psi^1\|_{\dot H^{s_c}}\geq
\frac18\left(\frac{A_1}2\right)^{\frac{N}{2s_c}}\left(\frac1{c^2_1}\right)^{\frac{N-s_c}{4s_c}}.$
Since $\|\chi\|_{\dot H^{-s_c}}\leq cr^{\frac N2-s_c}$, then
$\|\psi^1\|_{\dot H^{s_c}}\geq
\frac1{8c}\left(\frac{A_1}2\right)^{\frac{N}{2s_c}}\left(\frac1{c^2_1}\right)^{\frac{N-s_c}{4s_c}}r^{-(\frac
N2-s_c)}.$ In view of the choice of $r$, we obtain for some constant
$K$ such that \eqref{5.7'} holds, concluding the claim.

Let $W_n^1=\phi_n-e^{-it_n^1\Delta^2}\psi^1$. Then  since $e^{it_n^1\Delta^2}W_n^1\rightharpoonup0$ weakly in $H^2$,
for any $s\in[0,2]$,
$$<\phi_n,e^{-it_n^1\Delta^2}\psi^1>_{\dot H^{s}}=<e^{it_n^1\Delta^2}\phi_n,\psi^1>_{\dot H^{s}}
\rightarrow\|\psi^1\|^2_{\dot H^{s}}.$$
By expanding $\|W_n^1\|^2_{\dot H^{s}}$, we obtain
$$\lim_{n\rightarrow\infty}\|W_n^1\|^2_{\dot H^{s}}=\lim_{n\rightarrow\infty}\|\phi_n\|^2_{\dot H^{s}}
-\|\psi^1\|_{\dot H^{s}}.$$
From this with $s=0$ and $s=2$, we deduce that $\|W_n^1\|_{ H^{2}}\leq c_1.$

Let $A_2=\limsup_{n\rightarrow\infty}\|e^{it\Delta^2}W_n^1\|_
{L^\infty(\mathbb R;L^{\frac{2N}{N-2s_c}})}$. If $A_2=0$, then we are done.
If $A_2>0$, then we  repeat the above argument with $\phi_n$
replaced by $W_n^1$ to obtain a sequence of time shifts $t_n^2$ and
a profile $\psi^2\in H^2$ such that
$e^{it_n^2\Delta^2}W_n^1\rightharpoonup\psi^2$ weakly in $H^2$, and
$$\|\psi^2\|_{\dot H^{s_c}}\geq
K\left( A_2 \right)^{\frac{N}{2s_c}+\frac{N-2s_c}{min\{2s_c,4-2s_c\}}}\left(\frac1{c_1}\right)^{\frac{N-s_c}{2s_c}
+\frac{N-2s_c}{min\{2s_c,4-2s_c\}}}.$$
We show that $|t_n^1-t_n^2|\rightarrow\infty$. Indeed, if we suppose, up to a subsequence,
$t_n^2-t_n^1\rightarrow t_0$ finite, then
$$e^{i(t_n^2-t_n^1)\Delta^2}(e^{it_n^1\Delta^2}\phi_n-\psi^1)=e^{it_n^2\Delta^2}(\phi_n-e^{-it_n^1\Delta^2}\psi^1)=
e^{it_n^2\Delta^2}W_n^1\rightharpoonup\psi^2.$$
Since $e^{it_n^1\Delta^2}\phi_n-\psi^1\rightharpoonup0$, the left side of the above expression converges weakly to 0,
and so $\psi^2=0$, a contradiction. Let $W_n^2=\phi_n-e^{-it_n^1\Delta^2}\psi^1-e^{-it_n^2\Delta^2}\psi^2$.
Note that for any $s\in[0,2]$,
$$<\phi_n,e^{-it_n^2\Delta^2}\psi^2>_{\dot H^{s}}=<e^{it_n^2\Delta^2}\phi_n,\psi^2>_{\dot H^{s}}
=<e^{it_n^2\Delta^2}(\phi_n-e^{it_n^1\Delta^2}\psi^1),\psi^2>_{\dot H^{s}}+o_n(1)
\rightarrow\|\psi^2\|^2_{\dot H^{s}}.$$
We expand $$\lim_{n\rightarrow\infty}\|W_n^2\|^2_{\dot H^{s}}=\lim_{n\rightarrow\infty}\|\phi_n\|^2_{\dot H^{s}}
-\|\psi^1\|^2_{\dot H^{s}}-\|\psi^2\|^2_{\dot H^{s}},$$
and obtain $\|W_n^2\|_{ H^{2}}\leq c_1.$

We continue inductively, constructing a sequence $t_n^M$ and a profile $\psi^M$ such that
$e^{it_n^M\Delta^2}W_n^{M-1}\rightharpoonup\psi^M$ weakly in $H^2$, and
$$\|\psi^M\|_{\dot H^{s_c}}\geq
K\left( A_M \right)^{\frac{N}{2s_c}+\frac{N-2s_c}{min\{2s_c,4-2s_c\}}}\left(\frac1{c_1}\right)^{\frac{N-s_c}{2s_c}
+\frac{N-2s_c}{min\{2s_c,4-2s_c\}}}.$$
Suppose $1\leq j<M$. We show that $|t_n^M-t_n^j|\rightarrow\infty$ inductively by assuming that
$|t_n^M-t_n^{j+1}|\rightarrow\infty,\ldots, |t_n^M-t_n^{M-1}|\rightarrow\infty$.
In fact, Suppose up to a subsequence $t_n^M-t_n^j\rightarrow t_0$ finite.  We have
\begin{align*}
&e^{i(t_n^M-t_n^j)\Delta^2}(e^{it_n^j\Delta^2}W^{j-1}_n-\psi^j)
-e^{i(t_n^M-t_n^{j+1})\Delta^2}\psi^{j+1}-\cdots -e^{i(t_n^M-t_n^{M-1})\Delta^2}\psi^{M-1}\\
&=e^{it_n^2\Delta^2}(\phi_n-e^{it_n^1\Delta^2}\psi^1-\cdots-e^{-it_n^{M-1}\Delta^2}\psi^{M-1})=
e^{it_n^M\Delta^2}W_n^{M-1}\rightharpoonup\psi^{M}.
\end{align*}
Since the left side converges weakly to 0, then we get a contradiction since $\psi^{M}\neq0$.
This proves \eqref{divergence}. Let $W_n^M=\phi_n-e^{-it_n^1\Delta^2}\psi^1-e^{-it_n^2\Delta^2}\psi^2
-\cdots -e^{-it_n^M\Delta^2}\psi^M$.
Note that
 \begin{align*}
&<\phi_n,e^{-it_n^M\Delta^2}\psi^M>_{\dot H^{s}}=<e^{it_n^M\Delta^2}\phi_n,\psi^M>_{\dot H^{s}} \\
= &<e^{it_n^M\Delta^2}(\phi_n-e^{it_n^1\Delta^2}\psi^1-\cdots -e^{-it_n^{M-1}\Delta^2}\psi^{M-1}),\psi^M>_{\dot H^{s}}+o_n(1)\\
=&<e^{it_n^M\Delta^2}W_n^{M-1},\psi^M>_{\dot H^{s}}+o_n(1)
\rightarrow\|\psi^M\|^2_{\dot H^{s}},
\end{align*}
where the second line follows from the pairwise divergence property \eqref{divergence}.
The expansion \eqref{hsexpansion} is then shown by expanding $\|W_n^{M}\|^2_{\dot H^{s}}$.

By \eqref{hsexpansion} and $\|\psi^M\|_{\dot H^{s_c}}\geq
K\left( A_M \right)^{\frac{N}{2s_c}+\frac{N-2s_c}{min\{2s_c,4-2s_c\}}}\left(\frac1{c_1}\right)^{\frac{N-s_c}{2s_c}
+\frac{N-2s_c}{min\{2s_c,4-2s_c\}}}$, we get that
$$\sum_{M=1}^\infty \left(K\left( A_M \right)^{\frac{N}{2s_c}+\frac{N-2s_c}{min\{2s_c,4-2s_c\}}}\left(\frac1{c_1}\right)^{\frac{N-s_c}{2s_c}
+\frac{N-2s_c}{min\{2s_c,4-2s_c\}}}\right)^2\leq\lim_{n\rightarrow\infty}\|\phi_n\|^2_{\dot H^{s}}\leq c_1^2.$$
Since $N>2s_c$,  then  $A_M\rightarrow0$ as $M\rightarrow\infty$, which implies \eqref{remainder}.
\end{proof}

 \begin{lemma}\label{energy  expansion}(Energy pythagorean expansion)
 In the situation of Lemma \ref{lpd}, we have
 \begin{equation}\label{epe}
E(\phi_{n})=\sum_{j=1}^{M}E(e^{-it_{n}^{j}\Delta^2}\psi^{j})+E(W_{n}^{M})+o_n(1).
\end{equation}
\end{lemma}

\begin{proof}
According to \eqref{hsexpansion}, it suffices to establish for all $M\geq1$,
 \begin{equation}\label{nonlinear expansion}
\|\phi_{n}\|_{p+1}^{p+1}=\sum_{j=1}^{M}\|e^{-it_{n}^{j}\Delta^2}\psi^{j}\|_{p+1}^{p+1}+\|W_{n}^{M}\|_{p+1}^{p+1}+o_n(1).
\end{equation}
In fact, there are only two cases to consider.
Case 1. There exists some $j$ for which $t_n^j$ converges to a finite number, which, without loss of generality,
we assume is 0. In this case we will show that $\lim_{n\rightarrow\infty}\|W_n^M\|_{p+1}=0$  for $M>j$,
 $\lim_{n\rightarrow\infty}\|e^{-it_{n}^{k}\Delta^2}\psi^{k}\|_{p+1}=0$ for all $k\neq j$, and
 $\lim_{n\rightarrow\infty}\|\phi_n\|_{p+1}=\|\psi^{j}\|_{p+1}$, which gives \eqref{nonlinear expansion}.
 Case 2. For all $j$, $|t_n^j|\rightarrow\infty$. In this case we will show that
 $\lim_{n\rightarrow\infty}\|e^{-it_{n}^{k}\Delta^2}\psi^{k}\|_{p+1}=0$ for all $k$ and
 $\lim_{n\rightarrow\infty}\|\phi_n\|_{p+1}=\lim_{n\rightarrow\infty}\|W_n^M\|_{p+1}$, which  gives
 \eqref{nonlinear expansion} again.

 For Case 1, we infer from the proof of Lemma \ref{lpd} that $W_n^{j-1}\rightharpoonup\psi^j$.
 By the compactness of the embedding $H^2_{rad}\hookrightarrow L^{p+1}$, it follows that
$W_n^{j-1}\rightarrow\psi^j$ strongly in $L^{p+1}$. Let $k\neq j$. Then we get from \eqref{divergence}
that $|t_n^k|\rightarrow\infty$. As argued in the proof of Lemma \ref{lpd}, from Sobolev embedding
and the $L^q$ spacetime decay estimate of the linear flow, we obtain that
$\|e^{-it_{n}^{k}\Delta^2}\psi^{k}\|_{p+1}\rightarrow0$. Recalling that $$
W_n^{j-1}=\phi_n-e^{-it_{n}^{1}\Delta^2}\psi^{1}-\cdots -e^{-it_{n}^{j-1}\Delta^2}\psi^{j-1},$$
we conclude that $\phi_n\rightarrow\psi^j$ strongly in $L^{p+1}$. Since
$$W_n^M=
W_n^{j-1}-\psi^j-e^{-it_{n}^{j+1}\Delta^2}\psi^{j+1}-\cdots -e^{-it_{n}^{M}\Delta^2}\psi^{M},$$
we also conclude that $\lim_{n\rightarrow\infty}\|W_n^M\|_{p+1}\rightarrow0$ strongly in $L^{p+1}$, for $M>j$.

Case 2 follows similarly from the proof of Case 1.
\end{proof}

\begin{proposition}\label{critical}

There exists a radial $u_{c,0}$ in $H^2$ with
$$E(u_{c,0})=(M^{\frac{2-s_c}{s_c}}E)_c<M(Q)^{\frac{2-s_c}{s_c}}E(Q),\ \
\ \|\Delta u_{c,0}\|_2<\|Q\|^{\frac{2-s_c}{s_c}}_2\|\Delta Q\|_2$$
such that if $u_c$ is the corresponding solution of \eqref{1.1} with
the initial data $u_{c,0}$, then $u_c$ is global and $\|
u_c\|_{Z(\mathbb R)}=+\infty.$
\end{proposition}

\begin{proof}
Recall from Remark \ref{r5}, we have obtained  a radial sequence $u_n$ with
$\|u_n\|_2=1$  in the beginning of this section,  satisfying
 $\|\Delta u_{n,0}\|_{2}<\|Q\|^{\frac{2-s_c}{s_c}}_2\|\Delta Q\|_2$  and $E(u_n)\downarrow (M^{\frac{2-s_c}{s_c}}E)_c$ as
$n\rightarrow \infty.$ Each $u_n$ is global and nonscattering such
that $\| u_n\|_{Z(\mathbb R)}=+\infty.$ We  apply  Lemma \ref{lpd}
to $u_{n,0}$, which is uniformly bounded in $H^2$ to get
\begin{align}\label{5.7}
u_{n,0}(x)=\sum_{j=1}^{M}e^{-it_{n}^{j}\Delta^2}\psi^{j}(x)+W_{n}^{M}(x).
\end{align}
Then by \eqref{epe}, we have further
 \begin{equation*}
\sum_{j=1}^{M}\lim_nE(e^{-it_{n}^{j}\Delta^2}\psi^{j})+\lim_nE(W_{n}^{M})=\lim_nE(u_{n,0})=(M^{\frac{2-s_c}{s_c}}E)_c.
\end{equation*}
Since also by the profile expansion, we have
 \begin{equation*}
\|\Delta u_{n,0}\|_2^2=\sum_{j=1}^{M}\|\Delta e^{-it_{n}^{j}\Delta^2}\psi^{j}\|_2^2+\|\Delta W_{n}^{M}\|_2^2+o_n(1),
\end{equation*}
and
 \begin{equation}\label{5.9}
1=\|  u_{n,0}\|_2^2=\sum_{j=1}^{M}\| \psi^{j}\|_2^2+\|  W_{n}^{M}\|_2^2+o_n(1).
\end{equation}
Since  from the proof of Lemma \ref{4.5}, each energy is nonnegative
and then
\begin{align}\label{5.8}
\lim_nE(e^{-it_{n}^{j}\Delta^2}\psi^{j})\leq(M^{\frac{2-s_c}{s_c}}E)_c.
\end{align}
Now, if more than one $\psi^{j}\neq0$, we will show  a contradiction
 and thus the profile expansion will be reduced  to
the case that  only one  profile is not null.

In fact, if more than one $\psi^{j}\neq0$,  then by  \eqref{5.9} we must
have $M( \psi^{j})<1$ for each $j$, which together with \eqref {5.8}
implies that for $n$ large enough,
 \begin{equation*}
M(e^{-it_{n}^{j}\Delta^2}\psi^{j})^{\frac{2-s_c}{s_c}}E(e^{-it_{n}^{j}\Delta}\psi^{j})<(M^{\frac{2-s_c}{s_c}}E)_c.
\end{equation*}
For a given $j$, if $|t_n^j|\rightarrow+\infty$ we assume $t_n^j\rightarrow+\infty$
or $t_n^j\rightarrow-\infty$  up to a subsequence. In this case, by the proof of
Lemma \ref{energy  expansion}, we have $\lim_{n\rightarrow+\infty}\|e^{-it_{n}^{j}\Delta}\psi^{j}\|_{p+1}^{p+1}=0,$
and thus,
$\frac{1}{2}\| \psi^{j}\|_2^{\frac{2(2-s_c)}{s_c}}\|\Delta \psi^{j}\|_2^2=
\frac{1}{2}\|e^{-it_{n}^{j}\Delta^2} \psi^{j}\|_2^{\frac{2(2-s_c)}{s_c}}\|\Delta e^{-it_{n}^{j}\Delta^2}\psi^{j}\|_2^2<(M^{\frac{2-s_c}{s_c}}E)_c$.
If we denote by $NLS(t)\phi$   a solution of \eqref{1.1} with initial data $\phi$,
then we get  from the existence of wave operators ( Proposition \ref{wave operator} )that
there exists $\tilde{\psi}^{j}$ such that
$$\|NLS(-t_{n}^{j})\tilde{\psi}^{j}-e^{-it_{n}^{j}\Delta^2}\psi^{j}\|_{H^2}\rightarrow0,\ \ as\ \ n\rightarrow+\infty$$
with
$$\| \tilde{\psi}^{j}\|^{\frac{2-s_c}{s_c}}_2\|\Delta NLS(t)\tilde{\psi}^{j}\|_2<
\| Q\|^{\frac{2-s_c}{s_c}}_2\|\Delta Q\|_2$$
$$\| \tilde{\psi}^{j}\|_2=\| \psi^{j}\|_2,\ \ \ E(\tilde{\psi}^{j})=\frac{1}{2}\|\Delta \psi^{j}\|_2^2.$$
Thus we get that
 $M(\tilde{\psi}^{j})^{\frac{2-s_c}{s_c}}E(\tilde{\psi}^{j})<(M^{\frac{2-s_c}{s_c}}E)_c,$
 and so
$ \|NLS(t) \tilde{\psi}^{j}\|_{Z(\mathbb R)}<+\infty.$  If, on the
other hand, for the given $j$,  $t_{n}^{j}\rightarrow t'$ finite (at most only one such j by \eqref{divergence}),
then by the continuity of the linear flow in $H^2$, we have
$$e^{-it_{n}^{j}\Delta^2} \psi^{j}\rightarrow e^{-it'\Delta^2} \psi^{j} \ \ \ strongly \ \ in\ \ H^2.$$
In this case, we set $\tilde{\psi}^{j}=NLS(t')[e^{-it'\Delta^2} \psi^{j}]$ so that
 $NLS(-t')\tilde{\psi}^{j}=e^{-it'\Delta^2} \psi^{j}$.
To sum up,  in either case, we obtain a new profile $\tilde{\psi}^{j}$
for the given $\psi^{j}$  such that
\begin{align}\label{nl}
\|NLS(-t_{n}^{j})\tilde{\psi}^{j}-e^{-it_{n}^{j}\Delta^2}\psi^{j}\|_{H^2}\rightarrow0,\
\ as\ \ n\rightarrow+\infty.
 \end{align}
 As a result, we can replace
$e^{-it_{n}^{j}\Delta^2} \psi^{j}$ by
$NLS(-t_{n}^{j})\tilde{\psi}^{j}$ in \eqref{5.7} and obtain
\begin{align*}
u_{n,0}(x)=\sum_{j=1}^{M}NLS(-t_{n}^{j})\tilde{\psi}^{j}(x)+\tilde{W}_{n}^{M}(x),
\end{align*}
with $
\lim_{M\rightarrow+\infty}[\lim_{n\rightarrow+\infty}\|e^{it\Delta^2}\tilde{W}_{n}^{M}\|_{Z(\mathbb
R)}]=0.$

In order to use the perturbation theory to get a contradiction, we set
$v^j(t)=NLS(t)\tilde{\psi}^{j}$, $u_n(t)=NLS(t)u_{n,0}$ and set
$\tilde{u}_n(t)=\sum_{j=1}^{M}v^j(t-t_{n}^{j}).$
Then we have
$$i\partial_t\tilde{u}_n+\Delta^2\tilde{u}_n-|\tilde{u}_n|^{p-1}\tilde{u}_n=e_n,$$
where $$e_n=\sum_{j=1}^{M}|v^j(t-t_{n}^{j})|^{p-1}v^j(t-t_{n}^{j})-|\sum_{j=1}^{M}v^j(t-t_{n}^{j})|^{p-1}\sum_{j=1}^{M}v^j(t-t_{n}^{j}).$$
We will prove the following two claims to get the contradiction:\\
~Claim 1. There exists a large constant $A$ independent of $M$ such
that
for any $M$, there exists $n_0=n_0(M)$ such that for $n>n_0$, $\|\tilde{u}_n\|_{Z(\mathbb R)}\leq A.$\\
~Claim 2. For each $M$  and $\epsilon>0$,  there exists
$n_1=n_1(M,\epsilon)$ such that for $n>n_1$
$\|e_n\|_{Z'(I)}\leq \epsilon.$\\
Note that,   since
$\tilde{u}_n(0)-u_n(0)=\tilde{W}_{n}^{M},$ there exists
$M_1=M_1(\epsilon)$ such that for each $M>M_1$ there exists
$n_2=n_2(M)$ satisfying
$\|e^{it\Delta^2}(\tilde{u}_n(0)-u_n(0))\|_{Z(\mathbb R)}\leq
\epsilon$  with $\epsilon<\epsilon_0$ as in Lemma
\ref{perturb}. Thus, if the two claims hold true,  by the  long-time perturbation theory (Lemma
\ref{perturb}), for $n$ and $M$ large enough, we obtain
$\|u_n\|_{Z(\mathbb R)}<+\infty,$ which is a contradiction. So it
suffices to show the above claims.

Let $M_0$ sufficiently large such that
$\|e^{it\Delta^2}\tilde{W}_{n}^{M_0}\|_{Z(\mathbb R)}\leq\frac12
\delta_{sd}.$ Thus we get from the definition of
$\tilde{W}_{n}^{M_0}$ that for each $j>M_0$,
$\|e^{it\Delta^2}v^j(-t_{n}^{j})\|_{Z(\mathbb R)}\leq \delta_{sd}.$
Similar to the small data scattering and Proposition \ref{wave
operator}, we obtain
\begin{align}\label{5.10}
\|v^j(t-t_{n}^{j})\|_{Z(\mathbb R)}\leq  2
\|e^{it\Delta^2}v^j(-t_{n}^{j})\|_{Z(\mathbb R)}\leq 2\delta_{sd},\ \ \
 \|D^{s_c}v^j\|_{2}\leq c\|\psi^j\|_{\dot H^{s_c}}.
\end{align}
Thus by
the elementary  inequality
$$|(\sum_{j=1}^Ma_j)^p-\sum_{j=1}^Ma_j^p|\leq c_M\sum_{j\neq k}|a_j||a_k|^{p-1}, p>1, a_j\geq0$$
and \eqref{nl}, we have that
\begin{align}\label{5.11}
\|\tilde{u}_n\|^{\frac{(N+4)(p-1)}{4}}_{Z(\mathbb R)}
&=\sum_{j=1}^{M_0}\|v^j(t-t_{n}^{j})\|^{\frac{(N+4)(p-1)}{4}}_{Z(\mathbb
R)}
+\sum_{j=M_0+1}^{M}\|v^j(t-t_{n}^{j})\|^{\frac{(N+4)(p-1)}{4}}_{Z(\mathbb R)}+crossterms\\
\nonumber
&\leq\sum_{j=1}^{M_0}\|v^j\|^{\frac{(N+4)(p-1)}{4}}_{Z(\mathbb R)}
+c\sum_{j=M_0+1}^{M}\|\psi^{j}\|^{\frac{(N+4)(p-1)}{4}}_{\dot{H}^{s_c}}+crossterms.
\end{align}
In view of \eqref{divergence}, by taking $n_0$ large enough,
the $crossterms$ can be made bounded. On the other hand, by
\eqref{5.7} and Lemma \ref{lpd},
\begin{align}\label{5.12}
\|u_{n,0}\|_{\dot{H}^{s_c}}^2
&=\sum_{j=1}^{M_0}\|\psi^j\|_{\dot{H}^{s_c}}^2
+\sum_{j=M_0+1}^{M}\|\psi^j\|_{\dot{H}^{s_c}}^2
+\|W_n^{M}\|_{\dot{H}^{s_c}}^2+o_n(1),
\end{align}
which shows that the quantity
$\sum_{j=M_0+1}^{M}\|\psi^j\|_{\dot{H}^{s_c}}^{\frac{(N+4)(p-1)}{4}}$
is bounded independently of $M$ since
$\frac{(N+4)(p-1)}{4}>\frac{N(p-1)}{4}>2$. Above all, \eqref{5.11}
gives that $\|\tilde{u}_n\|_{Z(\mathbb R)}$
 is bounded independently of $M$ for $n>n_0$ large enough.  So the first claim holds true.

On the other hand, since
\begin{align*}
e_n=\left(|v^j(t-t_{n}^{j})|^{p-1}-
|\sum_{k=1}^{M}v^k(t-t_{n}^{k})|^{p-1}\right)\sum_{j=1}^{M}v^j(t-t_{n}^{j}),
\end{align*}
then if $p-1>1$, we estimate
\begin{align*}
|e_n|\leq c\sum\sum_{k\neq j}^{M}|v^j(t-t_{n}^{j})|
|v^k(t-t_{n}^{k})|\Big(|v^j(t-t_{n}^{j})|^{p-2}+|v^k(t-t_{n}^{k})|^{p-2}\Big);
\end{align*}
while if $p-1<1$,
\begin{align*}
|e_n|\leq c\sum\sum_{k\neq j}^{M}|v^j(t-t_{n}^{j})|
|v^k(t-t_{n}^{k})|^{p-1}.
\end{align*}
Since by \eqref{divergence}, for $j\neq k$, $|t_{n}^{j}-
t_{n}^{k}|\rightarrow+\infty$, then
 we obtain that $\|e_n\|_{Z(\mathbb R)}$ goes to zero as $n\rightarrow\infty$,
concluding the second claim.

Up to now, we have reduced the profile expansion to the case that
$\psi^1\neq0$, and $\psi^j=0$ for all $j\geq2$. We now begin to show
the existence of a critical solution. By \eqref{5.9} we have
$M(\psi^1)\leq1,$ and by \eqref{5.8} we have
$\lim_nE(e^{-it_{n}^{1}\Delta^2}\psi^{1})\leq(M^{\frac{2-s_c}{s_c}}E)_c$.
If $t_{n}^{1}$ converges and, without loss of generality,
$t_{n}^{1}\rightarrow 0$ as $n\rightarrow+\infty,$ we take
$\tilde{\psi}^1=\psi^1$ and then we have
$\|NLS(-t_{n}^{1})\tilde{\psi}^1-e^{-it_{n}^{1}\Delta^2}\psi^1\|_{H^2}\rightarrow0$
as $n\rightarrow+\infty.$ If, on the other hand,
$t_{n}^{1}\rightarrow +\infty,$ then by the proof of Lemma
\ref{energy  expansion}, we have
$\lim_{n\rightarrow+\infty}\|e^{-it_{n}^{1}\Delta^2}\psi^{1}\|_{p+1}^{p+1}=0,$
and so $$\frac{1}{2}\|\Delta
\psi^{1}\|_2^2=\lim_nE(e^{-it_{n}^{1}\Delta^2}\psi^{1})\leq(M^{\frac{2-s_c}{s_c}}E)_c.$$
Thus, by Proposition \ref{wave operator}, there exist
$\tilde{\psi}^1$ such that $M(\tilde{\psi}^1)=M(\psi^1)\leq1,$
$E(\tilde{\psi}^1)=\frac{1}{2}\|\Delta
\psi^{1}\|_2^2\leq(M^{\frac{2-s_c}{s_c}}E)_c,$ and
$\|NLS(-t_{n}^{1})\tilde{\psi}^1-e^{-it_{n}^{1}\Delta}\psi^1\|_{H^2}\rightarrow0$
as $n\rightarrow+\infty.$  In either case, if we set
$\tilde{W}_n^M=W_n^M+(e^{-it_{n}^{1}\Delta^2}\psi^1-NLS(-t_{n}^{1})\tilde{\psi}^1),$
then by Strichartz estimates, we have
$$\|e^{it\Delta}\tilde{W}_{n}^{M}\|_{Z(\mathbb R)}
\leq\|e^{it\Delta^2}W_{n}^{M}\|_{Z(\mathbb R)}+
c\|e^{-it_{n}^{1}\Delta^2}\psi^1-NLS(-t_{n}^{1})\tilde{\psi}^1\|_{\dot
H^{s_c}},$$ and
thus$$\lim_{n\rightarrow+\infty}\|e^{it\Delta^2}\tilde{W}_{n}^{M}\|_{Z(\mathbb
R)}
=\lim_{n\rightarrow+\infty}\|e^{it\Delta^2}W_{n}^{M}\|_{Z(\mathbb
R)}.$$ Therefore, we have
 $u_{n,0}=NLS(-t_{n}^{1})\tilde{\psi}^1+\tilde{W}_{n}^{M}$
with $M(\tilde{\psi}^1)\leq1,$
$E(\tilde{\psi}^1)\leq(M^{\frac{2-s_c}{s_c}}E)_c$ and
$\lim_{M\rightarrow+\infty}[\lim_{n\rightarrow+\infty}\|e^{it\Delta^2}\tilde{W}_{n}^{M}\|_{Z(\mathbb
R)}]=0.$ Let $u_c$ be the solution of \eqref{1.1} with initial data
$u_{c,0}=\tilde\psi^1$. We claim that $\|u_c\|_{Z(\mathbb
R)}=\infty,$ and then it must hold that $M(u_c)=1$ and
$E(u_c)=(M^{\frac{2-s_c}{s_c}}E)_c$ by the  definition of
$(M^{\frac{2-s_c}{s_c}}E)_c$, concluding the proof. Contrarily, if
otherwise
 $A\equiv\|NLS(t-t_{n}^{1})\tilde{\psi}^1)\|_{Z(\mathbb R)}=\|NLS(t)\tilde{\psi}^1)\|_{Z(\mathbb R)}
=\|u_c\|_{Z(\mathbb R)}<\infty.$ From the  perturbation theory
(Lemma \ref{perturb}), we get a $\epsilon_0=\epsilon_0(A).$ Taking $M$
sufficiently large and $n_2(M)$ large enough such that for $n>n_2$,
it holds that $\|e^{it\Delta^2}\tilde W_{n}^{M}\|_{Z(\mathbb R)}\leq
\epsilon_0.$ Similar to the proof of the  first case, Lemma
\ref{perturb} implies that there exists a large $n$ such that
$\|u_n\|_{Z(\mathbb R)}<\infty,$ which is a contradiction .

\end{proof}

\begin{proposition}\label{precompact}
(Precompactness of the flow of the critical solution) Let $u_c$ be
as  in Proposition \ref{critical}. If
$\|u_c\|_{Z([0,\infty))}=\infty$, then
$$K=\{u_c(t)| ~t\in[0,+\infty)\}\subset H^2$$
 is precompact in $H^2$. A corresponding conclusion is reached if
 $\|u_c\|_{Z((-\infty,0])}=\infty$.

\end{proposition}

\begin{proof}
Take a sequence $t_n\rightarrow+\infty$. We argue that $u_c(t_n)$
has a subsequence converging in $H^2$. In the sequel, we write $u=u_c$ for short.
Take
 $\phi_n=u(t_n)$ in the profile expansion lemma \ref{lpd}
to obtain profiles $\psi^j$ and a remainder $W_n^M$ such that
$$u(t_n)=\sum_{j=1}^M e^{-it_{n}^{j}\Delta^2}\psi^j+W_n^M$$
with $|t_{n}^{j}-t_{n}^{k}|\rightarrow+\infty$ as
$n\rightarrow+\infty$ for any $j\neq k$. Then Lemma \ref{energy
expansion} gives
$$\sum_{j=1}^M\lim_{n\rightarrow+\infty}E(e^{-it_{n}^{j}\Delta^2}\psi^j)
+\lim_{n\rightarrow+\infty}E(W_n^M)=E(u)=(M^{\frac{2-s_c}{s_c}}E)_c.$$
Similar to the proof of Lemma \ref{4.5}, we know that each energy
is nonnegative and thus for any $j$,
$$\lim_{n\rightarrow+\infty}E(e^{-it_{n}^{j}\Delta^2}\psi^j)\leq(M^{\frac{2-s_c}{s_c}}E)_c.$$
Moreover by \eqref{hsexpansion}, we have
$$\sum_{j=1}^MM(\psi^j)
+\lim_{n\rightarrow+\infty}M(W_n^M)=\lim_{n\rightarrow+\infty}M(u(t_n))=1.$$

If more than one $\psi^j\neq0.$ Similar to the proof in Proposition
\ref{critical}, we a contradiction from the definition of the
critical solution $u=u_c$. Thus we will address the case that only
$\psi^1\neq0$ and $\psi^j=0$ for all $j>1,$ and so
\begin{align}\label{5.13} u(t_n)=
e^{-it_{n}^{1}\Delta^2}\psi^1+W_n^M.
\end{align}
Also as in the proof of Proposition \ref{critical}, we obtain that
$M(\psi^1)=1$,
$\lim_{n\rightarrow+\infty}E(e^{-it_{n}^{1}\Delta^2}\psi^1)=(M^{\frac{2-s_c}{s_c}}E)_c,$
$\lim_{n\rightarrow+\infty}M(W_n^M)=0$ and
$\lim_{n\rightarrow+\infty}E(W_n^M)=0.$ Thus by Lemma \ref{4.5}, we
get
\begin{align}\label{2}
\lim_{n\rightarrow+\infty}\|W_n^M\|_{H^2}=0.
\end{align}

We claim now that $t_{n}^{1}$ converges to some finite $t^1$ up to a
subsequence, and then,  since $
e^{-it_{n}^{1}\Delta^2}\psi^1\rightarrow  e^{-it^{1}\Delta^2}\psi^1$
in $H^2$, \eqref{2} implies
 that $u(t_n)$ converges in $H^2$, concluding our proof.
It suffices to show the above claim. Contrarily, if
$t_{n}^{1}\rightarrow-\infty,$ then
$$\|e^{it\Delta^2}u(t_n)\|_{Z([0,\infty))}\leq
\|e^{i(t-t_{n}^{1})\Delta^2}\psi^1\|_{Z([0,\infty))}
+\|e^{it\Delta^2}W_n^M\|_{Z([0,\infty))}. $$
Note that
$$\lim_{n\rightarrow+\infty}\|e^{i(t-t_{n}^{1})\Delta^2}\psi^1\|_{Z([0,\infty))}
=\lim_{n\rightarrow+\infty}\|e^{it\Delta^2}\psi^1\|_{Z([-t_{n}^{1},\infty))}=0$$
and $\|e^{it\Delta^2}W_n^M\|_{Z([0,\infty))}\leq
\frac{1}{2}\delta_{sd} $ by taking $n$ sufficiently large, in contradiction to
the small data scattering theory. If, on the other hand,
$t_{n}^{1}\rightarrow+\infty,$ we will similarly have
 $\|e^{it\Delta^2}u(t_n)\|_{Z((-\infty,0])}\leq\frac{1}{2}\delta_{sd}. $
Thus  the small data scattering theory (Proposition \ref{sd}) shows
that  $\|u\|_{Z((-\infty,t_n])}\leq \delta_{sd}.$
Since
$t_n\rightarrow+\infty$, by sending $n\rightarrow+\infty$, we obtain
$\|u\|_{Z((-\infty,+\infty))}\leq \delta_{sd},$ which is a contradiction
again.

\end{proof}

\begin{corollary}\label{compact-localization}
Let $u$ be a solution of \eqref{1.1} such that $K=\{u(t)|
~t\in[0,+\infty)\}$ is precompact in $H^2$. Then for each
$\epsilon>0,$ there exists $R>0$ independent of $t$ such that
$$\int_{|x|>R}|\Delta u(x,t)|^2+|u(x,t)|^2+|u(x,t)|^{p+1}dx\leq\epsilon.$$
\end{corollary}
\begin{proof}
Contrarily, if not, then there exists $\epsilon_0>0$ and a sequence
$t_n\rightarrow+\infty$ such that, for each $n$,
$$\int_{|x|>n}|\Delta u(x,t_n)|^2
+|u(x,t_n)|^2+|u(x,t_n)|^{p+1}dx\geq2\epsilon_0.$$ By the
precompactness of $K$, there exists some $\phi\in H^2$ such that, up
to a subsequence of $t_n$, $u(t_n)\rightarrow\phi$ in $H^2$. Thus
taking  $n$ large, we obtain
\begin{align}\label{phi}
\int_{|x|>n}|\Delta \phi(x)|^2+|\phi(x)|^2
+|\phi(x)|^{p+1}dx\geq\epsilon_0. \end{align}
On the other hand,
since
 $\phi\in H^2$ and $\|\phi\|_{p+1}\leq
c\|\phi\|_{H^2}$ by Sobolev embedding, by taking $n$ sufficiently
large, we have  $$\int_{|x|>n}|\Delta \phi(x)|^2+|\phi(x)|^2
+|\phi(x)|^{p+1}dx\leq \frac12\epsilon_0,$$ in contradiction to
\eqref{phi}.

\end{proof}

\section{ A Rigidity theorem}
In this section, we prove the following statement and finish the proof of Theorem \ref{th1}.
\begin{theorem}\label{rigidity}
Assume $u_0\in H^2$ is radial satisfying $\|u_0\|_2=1$,
$E(u_0)<E(Q)M(Q)^{\frac{2-s_c}{s_c}}$ and $\|\Delta
u_{0}\|_{2}<\|Q\|^{\frac{2-s_c}{s_c}}_{2}\|\Delta Q\|_{2}.$ Let $u$
be the corresponding solution of \eqref{1.1} with initial data
$u_0$. If $K_+=\{u(t):t\in[0,\infty)\}$ is precompact in $H^2$, then
$u_0\equiv0$. The same conclusion holds if
$K_-=\{u(t):t\in(-\infty,0]\}$ is precompact in $H^2$.
\end{theorem}

\begin{remark}
In view of Proposition \ref{precompact}, Theorem \ref{rigidity}
implies that $u_c$ obtained in Proposition \ref{critical} cannot
exist. Thus, there must holds that $(M^{\frac{2-s_c}{s_c}}E)_c
=E(Q)M(Q)^{\frac{2-s_c}{s_c}}$, which combined with Proposition
\ref{h2scattering} implies Theorem \ref{th1}. Then, it suffices to show Theorem \ref{rigidity}.
\end{remark}

\begin{proof}
 Define
\begin{align}\label{action}
A_R(t)\equiv Im\int x\phi(\frac{|x|}{R})\cdot \nabla u\bar udx,
\end{align}
where $\phi(r)\in C_c^\infty(\mathbb R)$ is equal to 1 when $r \leq 1$,  and to 0 when $r\geq 2$.
Then for the radial solution $u$ of the  problem \eqref{1.1}, we compute that
\begin{align*}
A'_R(t)=Im\int\nabla\left(x\phi(\frac xR)\right)u_t\bar udx+2Im\int
x\phi(\frac xR)\cdot\nabla u\bar u_t\equiv I_1+I_2,
\end{align*}
where
\begin{align*}
I_1=&-Re\int\nabla\left(x\phi(\frac xR)\right)\bar
u(\Delta^2u-|u|^{p-1}u)dx \\=&-N\int|\Delta
u|^2-|u|^{p+1}dx+N\int\left(1-\phi(\frac xR)\right)\left(|\Delta
u|^2-|u|^{p+1}\right)dx\\&-Re\int\frac x R\cdot\nabla(\frac x
R)\left(\bar u\Delta ^2u-|u|^{p+1}\right)dx\\=&-N\int|\Delta
u|^2-|u|^{p+1}dx+Res_1
\end{align*}and
\begin{align*}
I_2=&-2Re\int x\phi(\frac xR)\cdot \nabla\bar
u(\Delta^2u-|u|^{p-1}u)dx\\=&-2Re\int x\phi(\frac xR)\cdot
\nabla\bar u\Delta^2udx-\frac2{p+1}\int\nabla\left(x\phi(\frac
xR)\right)|u|^{p+1}dx\\=&Re\int\nabla\left(x\phi(\frac
xR)\right)|\Delta u|^{2}dx-2Re\int\partial^2_k(\phi
x_j)\partial_j\bar u\Delta udx\\&-4Re\int\partial_k(\phi
x_j)\partial_j\partial_k\bar u\Delta
udx-\frac2{p+1}\int\nabla\left(x\phi(\frac
xR)\right)|u|^{p+1}dx\\=&N\int|\Delta u|^2dx+N\int(\phi(\frac
xR)-1)|\Delta u|^2dx+Re\int\frac xR\cdot\nabla\phi(\frac xR)|\Delta
u|^2dx\\&-2Re\int\partial^2_k(\phi x_j)\partial_j\bar u\Delta
udx-4\int|\Delta u|^2dx+4\int(\phi(\frac xR)-1)|\Delta
u|^2dx\\&-4Re\int\partial_k\phi(\frac xR)\partial
{x_j}R\partial_j\partial_k\bar u\Delta
u-\frac2{p+1}\int\nabla\left(x\phi(\frac
xR)\right)|u|^{p+1}dx\\=&(N-4)\int|\Delta
u|^2dx-\frac{2N}{p+1}\int|u|^{p+1}dx\\&+
\frac{2N}{p+1}\int(1-\phi)|u|^{p+1}dx -\frac 2{p+1}\int\frac
xR\cdot\nabla\phi|u|^{p+1}dx\\=&(N-4)\int|\Delta
u|^2dx-\frac{2N}{p+1}\int|u|^{p+1}dx+Res_2.
\end{align*}
From Corollary \ref{compact-localization}, we can infer that
$|Res|=|Res_1+Res_2|=o_R(1)\rightarrow0$ as $R\rightarrow\infty$,
uniformly in $t\in[0,\infty)$. Finally, we have
\begin{align}\label{virial}
A'_R(t)=-4\int|\Delta u|^2dx+\frac{N(p-1)}{p+1}\int|u|^{p+1}dx+o_R(1),
\end{align}
where $o_R(1)\rightarrow0$ as $R\rightarrow+\infty$ uniformly in $t$.

Let a positive constant $\delta\in(0,1)$ be such that
$E(u_0)<(1-\delta)E(Q)M(Q)^{\frac{2-s_c}{s_c}}$.
 By
Lemma \ref{lower bound} and Lemma \ref{4.5}, we obtain that there
exists some constant $\delta_0>0$ such that
$$-4\int|\Delta
u|^2dx+\frac{N(p-1)}{p+1}\int|u|^{p+1}dx\leq-2\delta_0\int|\Delta
u_0|^2dx,$$ which implies by \eqref{virial} that
$A'_R(t)\leq-\delta_0\int|\Delta u_0|^2dx$ for $R$ sufficiently
large. Thus, we have
$$A_R(t)-A_R(0)\leq-\delta_0t\int|\Delta
u_0|^2dx.$$ On the other hand, by the definition of $A_R(t)$, we should
have
$$|A_R(t)-A_R(0)|\leq C_R\|Q\|^2_{H^2},$$
which is a contradiction for $t$ large unless $u_0=0$.

\end{proof}


{\bf Acknowledgements:} The author was supported by the Project of NSFC for Young Scholars of China(Grant No.11301564).

\end{document}